\documentclass[11pt, a4paper,leqno]{amsart}
\usepackage{amsmath,amsthm,amscd,amssymb,amsfonts, amsbsy}
\usepackage{latexsym}
\usepackage{txfonts}
\usepackage{exscale}
\usepackage{graphicx}
\usepackage{bbm}
\usepackage{enumerate}
\usepackage[utf8]{inputenc}

\usepackage[colorlinks,citecolor=red,pagebackref,hypertexnames=false]{hyperref}

\usepackage{caption}
\usepackage{subcaption}

\usepackage{color}

\usepackage{tikz}
\tikzstyle{thin}=[line width=1.5pt]
\tikzstyle{fat}=[line width=2pt]
\tikzstyle{heavier}=[line width=3pt]
\tikzstyle{ultrafat}=[line width=4pt]
\tikzstyle{ultrafat-pt}=[line width=7pt]
\definecolor{myred}{HTML}{E53935}
\definecolor{myblue}{HTML}{1E88E5}
\definecolor{mygreen}{HTML}{43A047}
\definecolor{myyellow}{HTML}{FDD835}
\definecolor{myorange}{HTML}{FB8C00}
\definecolor{mygold}{HTML}{F9A825}
\definecolor{mypurple}{HTML}{8E24AA}
\definecolor{mygray}{HTML}{BDBDBD}
\definecolor{mybrown}{HTML}{6D4C41}
\definecolor{mynavy}{HTML}{1A237E}
\definecolor{mypink}{HTML}{ffbfca}
\definecolor{myseagreen}{HTML}{26A69A}
\definecolor{myviolet}{HTML}{f07ef0}
\definecolor{mydarkblue}{HTML}{0D47A1}
\definecolor{mydarkcyan}{HTML}{E0FFFF}
\definecolor{darkgray}{rgb}{0.66, 0.66, 0.66}
\definecolor{mydarkgreen}{HTML}{1B5E20}
\definecolor{mydarkmagenta}{HTML}{AD1457}
\definecolor{mydarkorange}{HTML}{EF6C00}
\definecolor{lightblue}{rgb}{0.68, 0.85, 0.9}
\definecolor{lightcyan}{rgb}{0.88, 1.0, 1.0}
\definecolor{lightgray}{rgb}{0.83, 0.83, 0.83}
\definecolor{mylightgreen}{HTML}{81C784}
\definecolor{lightyellow}{rgb}{1.0, 1.0, 0.88}
\definecolor{myshadow}{rgb}{0.5, 0.5, 0.5}
\definecolor{pink}{rgb}{1.0, 0.75, 0.8}
\definecolor{violet}{rgb}{0.93, 0.51, 0.93}
\definecolor{myauxcolor}{RGB}{245, 255, 255}
\definecolor{mylightgreen}{RGB}{193, 225, 159}

\let\oldtocsection=\tocsection
\let\oldtocsubsection=\tocsubsection
\let\oldtocsubsubsection=\tocsubsubsection
\renewcommand{\tocsection}[2]{\hspace{0em}\oldtocsection{#1}{#2}}
\renewcommand{\tocsubsection}[2]{\hspace{2em} \oldtocsubsection{#1}{\small{#2}}}
\renewcommand{\tocsubsubsection}[2]{\hspace{4em}\oldtocsubsubsection{#1}{\scriptsize{#2}}}

\calclayout
\allowdisplaybreaks

\theoremstyle{plain}
\newtheorem{theorem}[equation]{Theorem}
\newtheorem{lemma}[equation]{Lemma}
\newtheorem{corollary}[equation]{Corollary}
\newtheorem{proposition}[equation]{Proposition}

\theoremstyle{definition}
\newtheorem{definition}[equation]{Definition}

\theoremstyle{remark}
\newtheorem{remark}[equation]{Remark}

\numberwithin{equation}{section}

\newcommand{\dist}{\operatorname{dist}}

\newcommand{\re}{\mathbb{R}}
\newcommand{\rn}{\mathbb{R}^n}

\newcommand{\ree}{\mathbb{R}^{n+1}}

\newcommand{\N}{\mathbb{N}}

\newcommand{\cH}{\mathcal{H}}

\newcommand{\Xbf}{\mathbf{X}}
\newcommand{\xbf}{\mathbf{x}}
\newcommand{\Ybf}{\mathbf{Y}}
\newcommand{\ybf}{\mathbf{y}}
\newcommand{\Zbf}{\mathbf{Z}}
\newcommand{\zbf}{\mathbf{z}}

\DeclareMathOperator{\supp}{supp}
\DeclareMathOperator{\diam}{diam}

\newcommand{\vertiii}[1]{{\left\vert\kern-0.15ex\left\vert\kern-0.15ex\left\vert #1
		\right\vert\kern-0.15ex\right\vert\kern-0.15ex\right\vert}}

\def\Xint#1{\mathchoice
{\XXint\displaystyle\textstyle{#1}}%
{\XXint\textstyle\scriptstyle{#1}}%
{\XXint\scriptstyle\scriptscriptstyle{#1}}%
{\XXint\scriptscriptstyle%
\scriptscriptstyle{#1}}%
\!\int}
\def\XXint#1#2#3{{\setbox0=\hbox{$#1{#2#3}{%
\int}$ }
\vcenter{\hbox{$#2#3$ }}\kern-.6\wd0}}
\def\barint{\,\Xint -} 
\def\bariint{\barint_{} \kern-.4em \barint}
\def\bariiint{\bariint_{} \kern-.4em \barint}
\renewcommand{\iint}{\int_{}\kern-.34em \int} 
\renewcommand{\iiint}{\iint_{}\kern-.34em \int} 

\renewcommand{\d}{\, \mathrm{d}}

\newcommand{\R}{\mathbb{R}}
\newcommand{\Z}{\mathbb{Z}}
\providecommand{\ab}[1]{  \lvert  #1  \rvert }
\providecommand{\abs}[1]{ \left \lvert  #1 \right \rvert }

\providecommand{\no}[1]{  \lVert  #1  \rVert }

\providecommand{\trm}[1]{\textrm{#1}}

\DeclareMathOperator{\I}{I}
\DeclareMathOperator{\II}{II}

\title[]{The parabolic Dini-$\beta$ condition and absolute continuity of surface and caloric measure} 
\author[S. Bortz, M. Egert,  S. Ferris, O. Saari]
{Simon Bortz, Moritz Egert, Sandra Ferris and Olli Saari}

\address{Simon Bortz , 
Department of Mathematics and CONSERVE-AWI Group
\\
University of Alabama
\\
Tuscaloosa, AL, 35487, USA}

\email{sbortz@ua.edu}

\address{Moritz Egert, TU Darmstadt, Fachbereich Mathematik, Schlossgartenstr.\ 7, 64289 Darmstadt, Germany}
\email{egert@mathematik.tu-darmstadt.de}

\address{Sandra Ferris, Department of Mathematics
\\
University of Alabama
\\
Tuscaloosa, AL, 35487, USA}
\email{siferris.crimson@ua.edu}

\address{Olli Saari, Departament de Matem\`atiques,
	Universitat Polit\`ecnica de Catalunya,
	Avinguda Diagonal 647, 08028 Barcelona,
	Catalunya, Spain and Institute of Mathematics, Universitat Polit\`ecnica de Catalunya, Pau Gargallo 14, 08028 Barcelona, Catalunya, Spain}

\address{Centre de Recerca Matem\`atica, Edifici C, Campus Bellaterra, 08193 Bellaterra, Catalunya, Spain}
\email{olli.saari@upc.edu}

\date{\today}

\subjclass[2010]{}
\thanks{S.B. was  supported by the National Science Foundation (NSF DMS-2555449) a Simons foundation grant ``Travel support for Mathematicians'' and an AWI-CONSERVE fellowship.
O.S. was supported by the Spanish State Research Agency MCIN/AEI/10.13039/501100011033, Next Generation EU and by ERDF “A way of making Europe” through the grants RYC2021-032950-I, PID2021-123903NB-I00 and the Severo Ochoa and Maria de Maeztu Program for Centers and Units of Excellence in R\&D, grant number CEX2020-001084-M.
This project was carried out while the authors were visiting University of Alabama and the authors express their gratitude to UA. Roll Tide!}

\begin{document}
\allowdisplaybreaks

\begin{abstract}
We show if $\partial \Omega$ is the graph of a parabolic Lipschitz function,
then parabolic surface measure $\sigma$ of $\partial \Omega$ is absolutely continuous with respect to its caloric measure if and only if a (square) Dini-$\beta$ condition is satisfied. 
More specifically, the (square) Dini-$\beta$ condition is that
\[\int_0^1 \hat{\beta}(X,t,r)^2 \frac{dr}{r} < \infty, \quad \text{$\sigma$-a.e. } (X,t) \in \partial \Omega.\]
Here $\hat{\beta}$ is a parabolic version of the Jones ($L^2$) $\beta$-numbers. We show that these conditions are satisfied if and only if the graph is covered by a countable collection of {\it regular} Lipschitz graphs, that is, graphs with additional in-time regularity in the form of a half order time derivative in the parabolic BMO space. This supports the view that covering by {\it regular} Lipschitz graphs is the right notion for qualitative parabolic rectifiability in the context of parabolic PDEs. 

We also show that if 
\[\int_0^1 \hat{\beta}(X,t,r)^2 \frac{dr}{r} < \infty\]
up to a set of caloric measure zero then the caloric measure is absolutely continuous with respect to surface measure.
\end{abstract}

\maketitle 

\tableofcontents

\section{Introduction}
Throughout the paper, $n \ge 2$ is a natural number and a generic point of the space time is denoted by   
\[
\Xbf =  (x_0,x,t) = (x_0,\xbf) = (X,t)  \in \ree:=\re \times \re^{n-1}\times \re.
\]
Given a continuous function $f: \mathbb{R}^n \to \R$, $f = f(x,t)$, 
we study the heat equation in  
\[
\Omega_f := \{ \Xbf : x_0 > f(\xbf) \}.
\]  
\textit{The caloric measure} with pole at $\Xbf \in \Omega_f$ is the unique Radon measure on $\partial \Omega_f$ 
representing the Perron solutions of the continuous Dirichlet problem
\begin{align*}
\partial_t u - \Delta u &= 0 \quad \text{ in } \Omega_f, \\
u &= g  \quad \text{ on } \partial \Omega_f
\end{align*} 
as 
\[
u(\Xbf) = \int_{\partial \Omega} g(\ybf) d\omega_f^{\Xbf}(\ybf).
\]
The recent work \cite{BHMN1} together with \cite{Lew-Mur-Mem} show that the caloric measure is $A_{\infty}$ related in a suitable sense with the parabolic surface measure of the graph of a parabolic Lipschitz function $\partial \Omega_f$ if and only if $f$ is a \textit{regular} parabolic Lipschitz function (see Definition \ref{def:regular-graph}),
corresponding to the graph of a parabolic Lipschitz function being a parabolic uniformly rectifiable set.

While the papers \cite{BHMN1} and \cite{Lew-Mur-Mem} settle the question of \textit{quantitative} mutual absolute continuity, 
they leave open the question about absolute continuity understood as a \textit{qualitative} condition.

In this paper,
we introduce a condition in terms of beta numbers, corresponding to parabolic rectifiability,
which characterizes the \textit{qualitative} mutual absolute continuity of the parabolic surface measure and the caloric measure,
precisely understood as follows:
\begin{definition}\label{abscty.def}
We say the (parabolic) surface measure $\sigma_f$ of $\partial \Omega_f$ is absolutely continuous with respect to caloric measure $\omega_f$ (understood as the family of all caloric measures $\{\omega_f^{\Xbf}\}_{\Xbf \in \Omega_f}$) if for every $(X,t) \in \Omega$ it holds that
for all Borel sets $E \subset \partial\Omega_f$
\[\omega_f^{\Xbf}( \{(Y,s) \in E: s < t\}) = 0 \quad \Longrightarrow \quad  \sigma_f(\{(Y,s) \in E: s < t\})  = 0.\]
In this case we will write $\sigma_f \ll \omega_f$.

\end{definition}

We define the parabolic beta number for $\Xbf \in \partial \Omega_f$ and $r > 0$ as  
\begin{equation}
\label{def:beta-number-intro}
\hat{\beta}(f,\Xbf;r) = \inf_{S \in \mathcal{S}} \left( \fint_{Q(\xbf;r)} \left( \frac{\dist_p(\ybf,S)}{r} \right)^2 \, d \sigma_f(\ybf) \right)^{1/2}
\end{equation}
where $Q(\xbf;r)$ is ball with respect to the distance \eqref{parabolic-distance} 
and $\mathcal{S}$ is the family of all $n$-dimensional affine subspaces $S \subset \R^{n+1}$ 
for which $(X,t) \in S$ for some $t$ implies $(X,t) \in S$ for all $t \in \R$.
Our main theorem is the following:

\begin{theorem}\label{main.thrm}
Let $f: \R^{n} \to \R$ be a parabolic Lipschitz function. 
The following are equivalent:
\begin{itemize}
\item[(a)] $\sigma_f \ll \omega_f$ in the sense of Definition \ref{abscty.def}.
\item[(b)] $f$ satisfies the Dini-$\beta$ condition  
\begin{equation}
\label{Db.eq}
\int_0^1 \hat{\beta}(f,\Xbf;r)^2 \frac{dr}{r} < \infty, \quad \text{ for $\sigma_f$ almost every $\Xbf \in \partial \Omega_f$} .
\end{equation}
\item[(c)] There exists a family of \textit{regular} $Lip(1,1/2)$ functions $\{f_j\}_{j=1}^{\infty}$ such that 
\[
\sigma_{f}\left( \partial \Omega_f \setminus \bigcup_{j=1}^{\infty} \partial \Omega_{f_j} \right) = 0.
\]
\end{itemize} 
\end{theorem}

This problem has an extensive history. Following Dahlberg's celebrated result \cite{Dahl-L2}, where he showed the harmonic and surface measure were (quantitatively) mutually absolutely continuous in Lipschitz domains (with no additional assumptions), it was conjectured by Hunt (see \cite[p. 2]{KW-counter}) that the same should be true for parabolic Lipschitz domains. This was quickly disproven by Kaufman and Wu \cite{KW-counter}, where they showed for certain parabolic Lipschitz domains caloric, adjoint caloric and surface measure can all be singular with respect to one another. Here their construction involved a Weierstrass type function. Prior to the work of Kaufman and Wu, Wu \cite{Wu-split} had shown that if $E$ a subset of $\Gamma$ of zero surface measure then $E$ can be decomposed into two sets $E_\pm$ where $E_+$ is a null set of caloric measure and $E_-$ is a null set of adjoint caloric measure. In other words, null sets of surface measure are either inaccessible forward in time or backwards in time. These results illustrate the delicate interplay between the caloric measure and surface measure.

The first sufficient conditions to add to the function generating the graph were provided by Lewis, Murray and Silver \cite{Lew-Mur-Mem,Lew-Sil}. In \cite{Lew-Sil}, the closest classical work to our result, the authors study when $\Omega \subset \mathbb{R}^2$ is ``above" a (1/2)-H\"older function, $f(t)$. They show that $f(t)$ satisfies a ({\it global}) Dini condition on its modulus of continuity then surface measure and caloric measure are mutually absolutely continuous. This strong condition, when compared to Carleson type conditions described below, was inspired by \cite{CFK, FJK}. Nonetheless, in the spirit of \cite{CFK}, Lewis and Silver show the condition in \cite{Lew-Sil} is sharp in the sense that given a modulus of continuity, $\alpha$, which fails the Dini type condition there is a function $f'$ whose modulus of continuity, $\alpha'$, still fails their Dini condition, has $\alpha' \le \alpha$ and the caloric measure and surface measure are singular for the domain above the graph of $f'$. 

In the work of Lewis and Murray \cite{Lew-Mur-Mem}, the authors studied the {\it quantitative} absolute continuity of surface and caloric measure. In particular, Lewis and Murray showed that if $f$ as in Theorem \ref{main.thrm} satisfies the quantitative condition that it is a {\it regular} Lip(1,1/2) function (see Definition \ref{def:regular-graph}) then the caloric measure is an $A_\infty$ weight with respect to surface measure. In particular, this is equivalent to the existence of $p \in (1,\infty)$ so that the Dirichlet problem is solvable for $L^p$ data with accompanying $L^p$ estimates (on the non-tangential maximal function), see e.g. \cite{GH-Ainf}. Very recently, Hofmann, Martell and Nystr\"om and the first named author \cite{BHMN1} showed that the additional assumption that the function defining the graph is ``regular" is, in fact, necessary for this $L^p$-solvability. Our work should be thought of as a qualitative version of \cite{BHMN1} and Theorem \ref{main.thrm} (a) $\implies$ (b) relies heavily on the analysis there. 

We should also remark that Theorem \ref{main.thrm} provides further evidence that the ``correct" notion of parabolic rectifiablity, at least for the purposes of studying parabolic potential theory, should take into account Dini-type conditions like \eqref{Db.eq}. In fact, while writing this manuscript Hallgren, Koirala and Ma \cite{HKM-sing} showed that the singular set of caloric functions in $\mathbb{R}^{n+1}$ is contained in a countable collection of {\it regular} Lip(1,1/2) functions, up to a null set of $(n+1)$-dimensional parabolic Hausdorff measure. Their aims are considerably different than ours; however, the aspect of their proof which shows the graphs are regular is similar to our Theorem \ref{main.thrm} (c) $\implies$ (b) $\implies$ (a). This is because there is only one reasonable way to carry out this analysis, that is, following the arguments of David and Semmes \cite{DS-Ast} or the simpler but essentially equivalent argument in \cite{BHHLN-CME}. 

In light of the present manuscript and \cite{HKM-sing} a natural question to ask is whether one can prove as in \cite{Tolsa-partI, AT-partII} or \cite{ENV} that $(b) \iff (c)$ in significantly greater generality than the graphical setting. Moreover, one may also ask in what generality (b) or (c) $\implies$ (a), that is, parabolic analogues of \cite{AHMMMTV, ABaHM, ABHM}. The authors have not investigated these questions.

We also show that the proof that (b) $\implies$ (a) can be modified to prove the following theorem.

\begin{theorem}\label{partconv.thrm}
Let $f: \R^{n} \to \R$ be a parabolic Lipschitz function.  
If for some $\Ybf \in \Omega$
\[
\int_0^1 \hat{\beta}(f,\Xbf;r)^2 \frac{dr}{r} < \infty, \quad \text{for $\omega_f^{\Ybf}$ almost every $\Xbf \in \partial \Omega_f$},
\]
then $\omega_f^{\Ybf} \ll \sigma_f$.
\end{theorem}

The fact that Theorem \ref{partconv.thrm} is not purely geometric and takes into account the measure $\omega^{(Y,s)}$ is of course unfortunate, but common when trying to prove $\omega \ll \sigma$. Indeed, to the best of our knowledge results of this type (even in the elliptic setting) usually fall into two categories:
\begin{itemize}
\item  The geometric conditions placed on sets are so strong they imply quantitative absolute continuity in the form of an $A_\infty$-type condition, or
\item information of the form $\omega \ll \sigma$ can only be inferred on the common boundary of the original domain $\Omega$ with a suitable `nicer' approximating domain or surface.
\end{itemize}
Theorem \ref{partconv.thrm}, like the results of Azzam, Akman, Mourgoglou and Wu \cite{AAM,Wu-CADapprox} in the elliptic setting, fall into the second category. Results due to  Hofmann, Lewis, Nystr\"om, Murray and Silver \cite{HL-ann, HLN1, HLN2, Lew-Mur-Mem, Lew-Sil} or, more generally \cite{BHHLN-BPapprox}, which is inspired by the result of Nystr\"om and Str\"omqvist \cite{NS} and the elliptic result of David and Jerison \cite{DJ}, fall into the first category. 

The scheme of our the paper is as follows. In Section \ref{sec:notation}, 
we provide the necessary background definitions and reduce the theorem to a version stated in flat $\rn$. 
In Sections \ref{sec:a-to-c}, \ref{sec:b-to-a} and \ref{sec:c-to-b}, 
we prove the implications asserted in the coordinate version of Theorem \ref{main.thrm}, one in each.
The proof of Theorem \ref{partconv.thrm} is included in Section \ref{sec:b-to-a}. 
While the first implication,
presented in Section \ref{sec:a-to-c} is a streamlined presentation of the reasoning in \cite{BHMN1} and \cite{BFHH},
the other implications of the argument involve the use of the condition \eqref{Db.eq} and are the second contribution of this paper.

\section{Preliminaries}
\label{sec:notation}

\subsection{Notation}
We define the parabolic distance between two space-time points in $\mathbb{R}^{n+1}$ as
\begin{equation}
\label{parabolic-distance} 
|(X,t)-(Y,s)|  = \max(|X - Y|_{\infty},|t - s|^{1/2})
\end{equation}
and similarly between two space-time points in $\mathbb{R}^n$. 
Given $r > 0$ and $\Xbf \in \R^{n+1}$,
we denote a metric ball with respect to \eqref{parabolic-distance} by $Q(\Xbf,r)$.
We use the notations $\ell(Q(\Xbf,r)) = 2r$ and $c(Q(\Xbf,r)) = \Xbf$ for the sidelength and the center.
We use the same symbols for parabolic distances and cylinders in $\R^{n}$ when there is no risk of confusion. 
Unless explicitly stated otherwise, 
all further metric notions such as distances and diameters are understood with respect to the distance \eqref{parabolic-distance}.

\begin{definition}[Parabolic Hausdorff Measure]
Given a positive number $\eta$ and $\delta>0$,
we set for $E\subseteq \mathbb R^{n+1}$ 
\[ 
\cH_{\text{p},\delta}^\eta(E):= \inf \sum_k \diam(E_k)^\eta\,,
\]
where the infimum runs over all countable coverings of $E$, $\{E_k\}_{k \in \N}$, where the diameter with respect to the parabolic distance $\diam(E_k)\leq \delta$ for all $k$. We then define
\[
\cH_{\text p}^\eta (E) := \lim_{\delta\to 0^+} \cH_{\text{p},\delta}^\eta(E)\,.
\]
As for classical Hausdorff measure, $ \cH_{\text{p}}^\eta$ is a Borel regular measure.
We refer the reader to \cite[Chapter 2]{EG} for a discussion of the basic properties of standard
Hausdorff measure, which adapt readily to treat $  \cH_{\text{p}}^\eta$.
In particular, one obtains a measure equivalent to   $\cH_{\text{p}}^\eta$ if one defines
$\cH_{\text{p},\delta}^\eta$ in terms of coverings by arbitrary sets of parabolic diameter at most $\delta$, rather than cubes.
\end{definition}
	
\begin{definition}[Parabolic Lipschitz Functions]
\label{def:parabolic-lipschitz-f}
We say a function $f: \mathbb{R}^n \to \mathbb{R}$ is \textit{parabolic Lipschitz}, or $Lip(1,1/2)$, if there exists $L \ge 0$ such that
\[|f(x,t) - f(y,s)| \le L(|x - y| + |t - s|^{1/2}), \quad \forall (x,t), (y,s) \in \mathbb{R}^n.\]
If $f$ is parabolic Lipschitz we write $\|f\|_{Lip(1,1/2)}$ the infimum over all such $L$ and call this the parabolic Lipschitz or $Lip(1,1/2)$ constant of $f$.
\end{definition}

\begin{definition}[Parabolic Lipschitz Graph Domains]
\label{def:parabolic-lipschitz-g}
Given $f$ as in Definition \ref{def:parabolic-lipschitz-f},
we define the \textit{graph function of $f$}, $F: \mathbb{R}^n \to \mathbb{R}^{n+1}$, as
\[F(x,t) := (f(x,t),x,t).\] 
A set of the form $\partial \Omega_f = \{F(x,t): (x,t) \in \R^{n}\}$ is said to be a \textit{parabolic Lipschitz graph} if $f$ is parabolic Lipschitz.
The domain $\Omega_f \subset \R^{n+1}$ is a \textit{parabolic Lipschitz graph domain} if $\partial \Omega_f$ is a parabolic Lipschitz graph. 
\end{definition}

\begin{definition}[Parabolic Surface Measure]
\label{def:parabolic-lipschitz-sfm}
Given $f$ as in Definition \ref{def:parabolic-lipschitz-f},
we define the parabolic surface measure of $\partial \Omega_f$ as 
\[
\sigma_f := \mathcal{H}_{\text{p}}^{n+1}\rvert_{\partial \Omega_{f}}.
\]
\end{definition}

Finally,
relative to a reference parabolic cube $Q \subset \R^{n}$,
we define the Whitney region and Carleson box
\begin{align*}
W_f (Q) &= \{\Xbf \in Q(F(c(Q)),3L\ell(Q)): |\xbf - c(Q)| \le \ell(Q), \ x_0 - f(c(Q)) > 2L\ell(Q) \}, \\
\mathcal{R}_f(Q) &= \{\Xbf \in Q(F(c(Q)),3L\ell(Q)) \cap \Omega: |\xbf - c(Q)| \le \ell(Q) \}.
\end{align*} 
Moreover,
we know there exists a constant $A = A(n,L)$ such that for the good corkscrew point 
\[
{\trm crk}_f(Q) = F(c(A))  + (A \ell(Q),0,A \ell(Q)^{2})
\]
various Lemmas below (quoted from the literature) hold true.
Each of those references gives a lower bound on $A$ with the above mentioned dependencies,
and we may choose the largest of those.

\subsection{Half order time derivatives}
\label{subsec:half-order-time-derivative}
The half order time derivative $D_{t}^{1/2}$ of a function on $\R^{n}$ is defined as $\partial_t I$
where $I$ is a parabolic Riesz potential of order one, 
smoothing one order in space and a half order in time.
The potential $I$ has a convolution kernel $K(\xbf)$ whose Fourier transform satisfies 
\[
\widehat{K}(\xi,\tau) =   (\sqrt{|\xi|^{4} + 4\tau^{2}} + |\xi|^{2} )^{-1/2} \sim (|\xi| + |\tau|^{1/2})^{-1}
\]
so that by standard computations one may check $K$ is subject to the kernel estimates 
\[
|\partial^{\alpha}  \partial_t^{k} K(\xbf)| \le C_{\alpha,k}  (|x| + |t|^{1/2})^{-n-|\alpha|-2k}, 
\]
for $\alpha \in \N^{n-1}$, $k \in \N$ and $\xbf \in \R^{n}$.

\begin{definition}
\label{def:regular-graph}
A parabolic Lipschitz function $f : \R^{n} \to \R$ is said to be \textit{regular} if 
\[
\no{D^{1/2}_t f}_{BMO} := \sup_{\substack{\xbf \in \R^{n} \\ r >0}} \inf_{c \in \R} \fint_{Q(\xbf,r)}|D_{t}^{1/2}f(\ybf) - c| \, d \ybf < \infty.
\] 
\end{definition}

The regularity of a parabolic Lipschitz function can be tested by means of a condition involving $\beta$-numbers.
The following proposition is known to experts in the area (see the remark following it).
\begin{proposition}
Let $Q$ be a parabolic cube.
A parabolic Lipschitz function $f : \R^{n} \to \R$ with $\supp f \subset Q$ is \textit{regular} if and only if 
\[
M = \sup_{\substack{\xbf \in \R^n \\ r >0}} \frac{1}{r^{n+1}} \int_{0}^{r}\int_{Q(\xbf,r)} \check{\beta}(f,\ybf;r)^{2} \frac{d\ybf dr}{r} < \infty.
\]
\end{proposition}
\begin{remark}
It can be shown from \cite{HLN1,HLN2} (see \cite[Lemma 5.5]{BHHLN-corona}) that for $M$ above it holds
\[\no{D^{1/2}_t f}_{BMO} \le L + M^{1/2},\]
where $L$ is the parabolic Lipschitz constant of $f$. Deducing the bound 
\[M^{1/2} \le \no{D^{1/2}_t f}_{BMO} \le L\]
can be found in \cite[Section 5]{Hof-SIO}; however, we should note that in that manuscript Hofmann works with objects called $D_n$ and $D_{par}$, which are not $D_t^{1/2}$. In particular, in \cite[Section 5]{Hof-SIO} it is shown that $\| \mu_f \|_{\mathcal{C}}^{1/2}$ is bounded by $\|D_{par}\|_{BMO}$ up to a dimensional constant.  On the other hand, in \cite[Section 3]{Hof-SIO} it is observed that
\[D_{par}  = \sum_{j =1}^{n-1} \partial_{x_j}R_j + D_nR_n,\]
where $R_j$ are certain {\it parabolic} Riesz transforms, that is, the Fourier symbol of $R_j$ is  $\xi_j/\|(\xi, \tau)\|$ for $j \in \{1,2,\dots, n-1\}$ and the Fourier symbol of  $R_n$ is $\tau/\|(\xi, \tau)\|^2$ (here, the norm $\|\cdot\|$ is different, but equivalent to the one we use, so that the identity above holds). Since $R_j$ maps $BMO$ to $BMO$, it holds that
\[\|D_{par} f\|_{BMO} \lesssim \sum_{j = 1}^{n-1}\|\partial_{x_i} f\|_{BMO} + \|D_n f\|_{BMO} \lesssim \|\nabla  f\|_{L^\infty} + \|D_n f\|_{BMO}.\]
The final step is then to use 
\cite[Equation (0.19)]{HL-ann}, which states
\[ \|\nabla  f\|_{L^\infty} + \|D_n f\|_{BMO} \approx \|\nabla  f\|_{L^\infty} + \|D_t^{1/2} f\|_{BMO} .\] 
\end{remark}

\subsection{Reduction to flat coordinates}
\label{sec:reductions}
For the proof of Theorem \ref{main.thrm} and Theorem \ref{partconv.thrm},
we introduce the following objects related to a fixed parabolic Lipschitz function $f$.
We set for all $\Xbf_0 \in \Omega_f$
\begin{align}
\label{def:pullback-omega}
\omega^{\Xbf_0} &= \omega_{f}^{\Xbf_0} \circ F, \\
\label{def:pullback-sigma}
\sigma  &= \sigma_{f} \circ F .
\end{align}
Because $F$ is a bi-Lipschitz mapping $\R^{n} \to \partial \Omega_{f}$ relative to the parabolic distance,
we see that the zero sets of $\sigma$ are precisely the zero sets of $\sigma_{f}$.
Because of this,
absolute continuities $\omega_{f} \ll \sigma_f$ and $\sigma_{f} \ll \omega_{f}$ in the sense of Definition \ref{abscty.def}
are equivalent to the corresponding conditions stated in terms of $\omega^{\Xbf_0}$ and the Lebesgue measure on $\R^{n}$.
We also define the shifted versions of the Carleson and Whiteny regions (for a fixed $f$) as 
\begin{align*}
W (Q) &=  \{\Xbf: (f(\xbf)+x_0,\xbf) \in W_f(Q) \}, \\
\mathcal{R}(Q) &=  \{\Xbf: (f(\xbf)+x_0,\xbf) \in \mathcal{R}_f(Q) \}.
\end{align*} 

We define the coordinate beta numbers as 
\begin{equation}
\label{eq:def-coordinate-beta}
\check{\beta}(f,\xbf;r) = \inf_{A} \left( \fint_{Q(\xbf;r)} \left( \frac{|f(\ybf) - A(y)|}{r} \right)^2 \, d \ybf \right)^{1/2}
\end{equation}
where the infimum is taken over all affine mappings $A : \R^{n-1} \to \R$ with rational matrix and translation vector.
This function is measurable,
and we have the following relation between the different beta numbers.

\begin{proposition}
\label{prop:betas-pullback}
Let $f$ be parabolic $L$-Lipschitz.
There exists a dimensional constant $c_n$ such that 
\[
\hat{\beta}(f,F(\xbf);r) \le \check{\beta}(f,\xbf;r) \le (1 + c_n L) \hat{\beta}(f,F(\xbf);r)
\]
for all $\xbf \in \R^{n}$ . 
\end{proposition}
\begin{proof}
Fix $\xbf_0$ and $r_0$.
Without loss of generality we may assume $F(\xbf_0)=0$.

We first show $\check{\beta}(f,\xbf;r) \lesssim \hat{\beta}(f,F(\xbf);r)$.
Considering the plane
\[
H_{p,b} = \{\Xbf \in \R^{n+1}: x_0 = p \cdot x + b \}
\]
with $p=0$ and $b = 0$,
we see that  
\[
\check{\beta}(f,\xbf;r) \le L . 
\]
On the other hand,
for any plane $H_{p,b}$ with $p \in \R^{n-1}$ and $b \in R$,
we define the affine function $A_{p,b}(x) = p \cdot x + b$ 
and set 
\[
Q_p^{+} =  \{\xbf \in Q(0,1): x \cdot p > |p|/4 \}, \quad Q_p^{-} =  \{\xbf \in Q(0,1): x \cdot p < |p|/4 \} .
\]
Then 
\[
\inf_{\xbf,\ybf \in Q_p} |A_{p,b}(\xbf) - A_{p,b}(\ybf)| \ge |p|/2
\]
and we see that for all $\ybf \in Q_p^{\pm}$ and $\xbf \in Q_{p}^{\mp}$
\[
|f(\ybf) - A_{p,b}(y)| \ge |A_{p,b}(y)- A_{p,b}(x)| - |f(\ybf) - A_{p,b}(x)| \ge |p|/2 - L - |f(\xbf) - A_{p,b}(x)|
\]
so that if $|f(\xbf) - A_{p,b}(x)| \le |p|/2 - L$ for some $\xbf \in Q_p^{\pm}$,
then for all $\ybf \in Q_p^{\mp}$
\[
|f(\ybf) - A_{p,b}(y)| \ge |p|.
\]
Consequently, there exists $c_n$ large enough so that for $|p| \ge c_n L$
\[
\fint_{Q(0,r)} \frac{|f(\ybf) - A_{p,b}(y)|^{2}}{r^{2}} \, d \ybf \ge \frac{|Q_p^{\pm}|}{|Q(0,1)|} |p|^{2}
\]
and no such $H_{p,b}$ will do as a competitor against $H_{0,0}$.  

To conclude we need to prove a lower bound for $\dist(F(\ybf),H_{p,b})$
with $|p| \le c_n L$ and all $\ybf \in Q(0,r)$. 
Let $Q = Q(F(\ybf),R)$ be the maximal parabolic cube centred at $F(\ybf)$ whose interior does not meet $H_{p,b}$.
If $R = 0$, then $\check{\beta}(f,\xbf;r) = \hat{\beta}(f,F(\xbf);r) = 0$ and the claim holds.
If $R > 0$,
we pick $\Zbf \in \overline{Q} \cap H_{p,b}$ and compute 
\[
R 
= |\Zbf - F(\ybf)| 
= |(A_{p,b}(z) - f(\ybf),\zbf-\ybf)| 
\ge |A_{p,b}(y) - f(\ybf)| - |A_{p,b}(y) - A_{p,b}(z)| 
\ge |A_{p,b}(y) - f(\ybf)| - c_n L R .
\]  
Hence 
\[
\dist(F(\ybf),H_{p,b}) \ge 
|f(\ybf) - A_{p,b}(y)|/(1+c_n L) .
\]
As $\ybf$ was arbitrary,
we see that $\check{\beta}(f,\xbf;r) \le \hat{\beta}(f,F(\xbf);r)$ holds.

To prove the other direction,
it is enough to note that trivially 
\[
\overline{Q(F(\ybf),|A_{p,b}(y) - f(\ybf)|)} \cap H_{p,b} \ni \{(A_{p,b}(y), \ybf) \} \ne  \varnothing 
\]
even without any constraint on $p$ and $b$.
Consequently, $|A_{p,b}(y) - f(\ybf)| \ge \dist(F(\ybf), H_{p,b})$ and further $\check{\beta}(f,\xbf;r) \ge \hat{\beta}(f,F(\xbf);r)$ as claimed.
\end{proof}

Due to Proposition \ref{prop:betas-pullback} and the discussion above it in this subsection,
we see that Theorem \ref{main.thrm} is equivalent to the following coordinate version
which we will prove below.

\begin{theorem}\label{main.thrm-coord}
Let $f: \R^{n} \to \R$ be a parabolic Lipschitz function. 
The following are equivalent:
\begin{itemize}
\item[(a)] For all $(X,t) \in \Omega_f$, the Lebesgue measure on $\{(y,s) \in \R^{n}: s < t\}$ is absolutely continuous with respect to $\omega_{f}^{(X,t)} \circ F$.
\item[(b)] $f$ satisfies the Dini-$\beta$ condition  
\begin{equation*} 
\int_0^1 \hat{\beta}(f,\Xbf;r)^2 \frac{dr}{r} < \infty, \quad \text{ for $\sigma_f$ almost every $\Xbf \in \partial \Omega_f$} .
\end{equation*}
\item[(c)] There exists a family of \textit{regular} $Lip(1,1/2)$ functions $\{f_j\}_{j=1}^{\infty}$ such that 
\[
 \abs{F^{-1}\left(\partial \Omega_f \setminus \bigcup_{j=1}^{\infty} \partial \Omega_{f_j} \right)} = 0.
\]
\end{itemize} 
\end{theorem}

\section{Absolute continuity implies exhaustion} 
\label{sec:a-to-c}
To prove that a $Lip(1,1/2)$ graph can be covered by countably many regular $Lip(1,1/2)$ graphs,
it suffices to show that for every $\varepsilon > 0$,
we may cover $(1-\varepsilon)$-portion of it by a regular $Lip(1,1/2)$ graph.
Such a regular $Lip(1,1/2)$ graph is constructed as a sawtooth type domain whose boundary has an ample contact with the original $Lip(1,1/2)$ graph.
The set $E$ below,
completely general for the time being,
will take the role of those coordinates for which $F(E)$ will be the contact set.
The regularized distance $h_E$ below is needed as the 
standard parabolic distance fails to be regular $Lip(1,1/2)$ function and hence is of little use for us.

\subsection{Regularized distance}
Let $E \subset \R^{n}$ be closed.
Consider the family of parabolic dyadic cubes 
\[
\mathcal{D}_k = \{ 2^{k}([0,1)^{n-1} + j')  \times  [ 2^{2k}j, 2^{2k}(j+1))  : j \in \Z, \ j' \in \Z^{n-1} \}.
\]
Let 
\[
\mathcal{W}_k = \{Q \in \mathcal{D}_k: 2^{k+2} \le  \dist(Q,E) < 2^{k+3}  \}
\]
and let $\mathcal{W}$ be the family of maximal elements in $\bigcup_{k \in \Z} \mathcal{W}_k$.
Let $\{\varphi_Q: Q \in \mathcal{W}\}$ be a smooth partition of unity adapted to the cover $\{3Q: Q \in \mathcal{W}\}$,
meaning that $\supp \varphi_{Q} \subset 3Q$ and for each $\alpha \in \N^{n}$ there exists $C_{\alpha}$ such that 
\[
\sup_{\xbf \in 3Q} |\partial^{\alpha} \varphi_Q(\xbf)| \le C_{\alpha} \ell(Q)^{-\alpha_{n}-|\alpha|}.
\]
The regularized distance to $E$ is defined for $x \in \R^{n-1}$ and $t \in R$ as 
\begin{equation}
\label{def:reg-distance}
h_E(x,t) := \sum_{Q \in \mathcal{W}}  \diam(Q) \varphi_{Q}(x,t). 
\end{equation}

The following lemma is essentially a restatement of Lemma 3.24 in \cite{BHHLN-CME}.
For a proof, see the appendix of \cite{BHHLN-CME}.
\begin{lemma}
\label{lemma:estimates-reg-dist}
Let $E \subset \R^{n}$ be closed
and let $h_E$ be as defined in \eqref{def:reg-distance}.
Then  
\begin{itemize}
  \item there is a constant $C_h$ depending only on the dimension such that for all $\xbf \in \R^{n}$ it holds $C_h^{-1} \dist(\xbf,E) \le  h_E(\xbf) \le C_{h}\dist(\xbf,E)$;
  \item $h_E$ is smooth in $E^{c}$ and $Lip(1,1/2)$ in $\R^{n}$ with constant depending only on the dimension;
  \item for all integers $m \ge 0$
\[
\sup_{\xbf \in E^{c}} \left(  h_{E}(\xbf)^{2m-1} |\partial_{t}^{m} h_E(\xbf)| + h_{E}(\xbf)^{m-1} |\nabla_{x}^{m} h_E(\xbf)|  \right) \lesssim 1
\]  
the implicit constant only depending on the dimension.
  \item it holds 
\[
\no{ D_t^{1/2} h_E}_{BMO} <  \infty,
\]
that is, $h_E$ is a regular $Lip(1,1/2)$ function.
\end{itemize} 
\end{lemma}

\subsection{Graphs from the Green's function}
When constructing the approximating regular $Lip(1,1/2)$ graph,
we aim to mimic the shape of 
\[
f(\xbf) + h_E(\xbf)
\]
for a suitable closed set $E$.
Our goal is to find a set $E$ such that $f$ is very regular in $E$
and the sawtooth type domain over the epigraph of $f$ allows us to replace the expression above by something more tractable.
To this end,
we will find a substituting graph following the level surfaces of the Green's function. 
The gradient of the Green's function always points more or less to $e_0$ direction,
as the following lemma states.

\begin{lemma}
\label{lemma:green-non-nondegenerate}
Let $f$ be parabolic $L$-Lipschitz. 
Let $Q_0 \subset \R^{n}$.
Let $\Xbf_0 = {\trm crk}_f(Q_0)$. 
Let $G$ be the adjoint Green's function with pole at $\Xbf_0$ and $\Xbf \in \mathcal{R}_f(Q_0)$.
Then  
\[
|\nabla G(\Xbf)| \sim \nabla G(\Xbf) \cdot e_0 \sim \frac{G(\Xbf)}{\dist(\Xbf,\Omega^{c})}.
\]
The implicit constants depend only on the dimension and $L$. 
\end{lemma}

This non-degeneracy of the Green's function implies that the level surfaces of the Green's function can indeed be seen as graphs of functions.

\begin{lemma}
\label{lemma:implicit}
Let $f$ be parabolic $L$-Lipschitz. 
Let $Q_0 \subset \R^{n}$.
Let $\Xbf_0 = {\trm crk}_f(Q_0)$.
Assume that $\omega^{\Xbf_0}(Q_0) > 0$.
Let $G$ be the adjoint Green's function with pole at $\Xbf_0$.

Then there exists $r_0 > 0$ such that for each $r \in (0,r_0)$
there exists $f(r; \cdot) : Q_0 \to \R$ such that 
\begin{itemize}
  \item it holds $G(f(r;\xbf),\xbf) = r $ for all $\xbf \in Q_0$;
  \item for each $\xbf \in Q_0$, the mapping $r \mapsto f(r;\xbf)$ is strictly increasing, differentiable and satisfies $\lim_{r \to 0+} f(r;\xbf) = f(\xbf) =: f_0(\xbf)$;
  \item the function $f(r;\cdot)$ is in $C^{1}(Q_0)$. 
\end{itemize}
\end{lemma}
Lemma \ref{lemma:implicit} is a restatement of Lemma 4.24 in \cite{BHMN1}.
The proof is based on the observation that by Lemma \ref{lemma:green-non-nondegenerate} 
the vertical derivative $\partial_{x_0} G(\Xbf) > 0$ for all $\Xbf \in \mathcal{R}_L(Q_0)$.
Due to this, one may apply the implicit function theorem to prove the existence of $f(r;\cdot)$ along with the properties listed above. 
For details, we refer to Lemma 4.24 in \cite{BHMN1}. 

For a closed set $E \subset Q_0$,
we can now define 
\begin{equation}
\label{eq:psi-graph-final-def}
\psi_{E,Q_0} (\Xbf) = \varphi_{Q_0}(\xbf) f( x_0 + h_E(\xbf); \xbf)
\end{equation}
so that $\psi_{E,Q_0} (\xbf):= \psi_{E,Q_0} (0,\xbf)$ is the candidate for a regular $Lip(1,1/2)$ graph approximating $f$. 
Here $\phi_{Q_0} \in C_{c}^{\infty}(2^{-1}Q_0)$ satisfies $\phi_{Q_0} = 1$ in $2^{-2}Q_0$. 
A function $\psi_E$ as above enjoys the enhanced regularity of $h_E$
and also that of the Green's function when $\xbf \notin E$.
However, 
in order to say more
we need to choose the set $E$ carefully.

\subsection{M-good Green subdomains}
The Green's function is not only qualitatively but also quantitatively non-degenerate 
inside the sawtooth domains of the following type.

\begin{definition}
\label{def:m-good-green}
Let $M,L \ge 1 $.
Let $f$ be parabolic $L$-Lipschitz. 
Let $Q_0 \subset \R^{n}$.
Let $\Xbf_0 = {\trm crk}_f(Q_0)$. 
A subdomain $\Omega' \subset \mathcal{R}_f(Q_0)$ is an $M$-good Green subdomain for $\Omega_f$ if 
for the adjoint Green's function $G$ of $\Omega_f$ with pole at $\Xbf_0$ and all $\Xbf \in \Omega'$ satisfies
\[
 \frac{1}{M} \le \frac{G(\Xbf)}{ \dist(\Xbf,\Omega_f^{c})} \frac{|Q_0|}{\omega_{f}^{\Xbf_0}(Q_0)} \le M .
\]
\end{definition}

Note that $M$-goodness is a non-trivial property even in the case of time-independent domains and the Laplacian.
It does not hold in any open neighborhood of a corner point of the domain.  
Next, given a closed set $E \subset Q_0$ and a slope parameter $\theta_0 > 0$,
we denote 
\[
S_{E,\theta_0} = \{ \Xbf \in \mathcal{R}_f(Q_0):  x_0 > f(\xbf) + \theta_0 h_E(\xbf) \}.
\]
If such a set is $M$-good for some slope,
it is good for any slope.
This follows from the time-symmetric Harnack inequality for Green's function.

\begin{lemma}
\label{lemma:green-backwards-harnack}
Let $f$ be parabolic $L$-Lipschitz. 
Let $Q_0 \subset \R^{n}$.
Let $\Xbf_0 = {\trm crk}_f(Q_0)$.
Let $G$ be the adjoint Green's function for $\Omega_f$ with pole at $\Xbf_0$.
Then for every $\eta \in (0,1)$ there exists a constant $C = C(n,\eta,L)$ 
such that fo all $\Xbf,\Ybf \in \mathcal{R}_f(Q_0)$ with $\dist(\Xbf, \Omega_f^{c}) > \diam(Q_0)\eta$ 
\[
G(Y) \le C G(X).
\]
\end{lemma} 

\begin{corollary}
\label{lemma:m-goodness-and-slope}
Let $f$ be parabolic $L$-Lipschitz. 
Let $Q_0 \subset \R^{n}$.
Let $\Xbf_0 = {\trm crk}_f(Q_0)$.
Let $E \subset Q_0$ be a closed set,
let $\theta_0 > 0$,
let $M \ge 1$ and let 
\[
S_{E,\theta_0} = \{ \Xbf \in \mathcal{R}_f(Q_0):  x_0 > f(\xbf) + \theta_0 h_E(\xbf) \}
\]
be an $M$-good Green subdomain.

Then for each $\theta \in (0,\theta_0)$
there exists $\tilde{M} = \tilde{M}(n,M,\theta_0,\theta) \ge 0$ such that $S_{E_{\theta}}$ is a $\tilde{M}$-good Green subdomain.
\end{corollary}
\begin{proof}
This follows immediately from the backward Harnack inequality for the Green's function (Lemma \ref{lemma:green-backwards-harnack}) and the definition of $M$-good Green subdomains (Definition \ref{def:m-good-green}).
\end{proof}

The sets $S_{E,\theta_0}$ that are $M$-good Green subdomains enjoy the following estimate.
The following lemma can be found as Lemma 8.37 in \cite{BFHH}.
We remark that reference deals with a slightly different function
\[
\tilde{\psi}_{E,Q_0} (\Xbf) = \varphi_{Q_0}(\xbf) f( x_0 + P_{\gamma x_0} h_E(\xbf); \xbf)
\]
with a Littlewood--Paley operator $P_{\gamma x_0}$,
but the same proof works for $\psi_{E,Q_0}(\Xbf)$ here. 
On the other hand, 
we advise the reader looking for the most transparent presentation
to start with the global square function estimate in the case of heat equation in \cite{BHMN1}.

\begin{lemma} 
\label{lemma:square-function} 
Let $f$ be parabolic $L$-Lipschitz. 
Let $Q_0 \subset \R^{n}$.
Let $\Xbf_0 = {\trm crk}_f(Q_0)$.
Let $E \subset Q_0$ be a non-empty closed set,
let $\theta_0 > 0$,
let $M \ge 1$ and let $S_{E,\theta_0}$ be an $M$-good Green subdomain.

Let $\psi_{E,Q_0}$ be as in \eqref{eq:psi-graph-final-def}.
Then there exists 
a constant $C = C(n,L,M,\theta_0
)$ such that 
for every parabolic subcube $Q \subset Q_0$ with $Q \cap E \ne \varnothing$
\[
\frac{1}{ |Q|} \int_{0}^{\diam(Q)/(8MC_h)} \iint_{Q} 
\left(
|r \partial_{t} \psi_{E,Q_0}(r,\xbf)|^{2}
+ |r \partial_r ^{2}  \psi_{E,Q_0}(r,\xbf)|^{2}
+ |r^{2} \partial_r   \partial_{t}    \psi_{E,Q_0}(r,\xbf)|^{2}
\right) \, \frac{d\xbf dr}{r}
\le C
\]
and for all $r \in (0,\diam(Q)/(8MC_h))$ and $\xbf \in Q$
\[
|r \partial_{t} \psi_{E,Q_0}(r,\xbf)|
+ |r \partial_r ^{2}  \psi_{E,Q_0}(r,\xbf)|
+ |r^{2} \partial_r   \partial_{t}    \psi_{E,Q_0}(r,\xbf)| \le C.
\] 
Here $C_h$ is the dimensional constant from Lemma \ref{lemma:estimates-reg-dist}.
\end{lemma}

\subsection{Bounded mean oscillation estimate}
Under the assumptions of Lemma \ref{lemma:square-function}
a half-order time derivative of $\psi_{E,Q_0}$ is in $BMO$.

\begin{lemma}
\label{lemma:bmo-estimate}
Let $f$ be parabolic $L$-Lipschitz. 
Let $Q_0 \subset \R^{n}$.
Let $\Xbf_0 = {\trm crk}_f(Q_0)$.
Let $E \subset Q_0$ be a non-empty closed set,
let $\theta_0 > 0$,
let $M \ge 1$ and let $S_{E,\theta_0}$ be an $M$-good Green subdomain.
Let $\psi_{E,Q_0}$ be as in \eqref{eq:psi-graph-final-def}.

Then 
\[
\no{ D_t^{1/2} \psi_{E,Q_0}}_{BMO} \le C(n,L,M,\theta_0) .
\]
\end{lemma}
\begin{proof} 
For a parabolic cube $Q \subset \rn$ with center $z_Q$.
Let $\tilde{Q} = Q - z_Q$. 
Let $\varphi_{\tilde{Q}} \in C_c^{\infty}(2Q)$ be such that $1_{\tilde{Q}} \le \varphi_{\tilde{Q}} \le 1$ 
and $|\partial^{\alpha} \partial_t^{k} \varphi_{\tilde{Q}}(\xbf)| \le C_{\alpha,k} \diam(Q)^{-|\alpha|-2k} $ for all $\alpha \in \N^{n-1}$ and $k \in \N$.

Let $K$ be the kernel of the half-order time derivative from Subsection \ref{subsec:half-order-time-derivative}.
Splitting the kernel as $K = K_Q  + K_2$ with $K_{Q} =  \varphi_{\tilde{Q}} K$,
we may split the parabolic potential operator into a localized part $I_Q$ with local and singular kernel $K_Q$
and the remainder part $I_2$ with non-local but smooth kernel $K_2$.
The kernel $k_2$ of $\partial_t I_2$ is an $L^{1}(\rn)$ function that has integral zero and that satisfies 
\[
\int |k_2(\xbf)| \, d\xbf
\lesssim \int  \partial_t [(1 - \varphi_{\tilde{Q}}(\ybf) ) K(\ybf)] \, d\ybf \lesssim \frac{1}{\diam(Q)}
\]
so that 
\[
\frac{1}{|Q|} \int_Q | \partial_t I_2 \psi_{E,Q_0}(\xbf) - (\partial_t I_2 \psi_{E,Q_0})_{Q}| \, d\xbf
\le 2\no{\partial_t I_2 \psi_{E,Q_0}}_{L^{\infty}(Q)} 
\lesssim \frac{\no{\psi_{E,Q_0}- \psi_{E,Q_0}(z_Q)}_{L^{\infty}(3Q)}}{\diam(Q)}
\lesssim 1. 
\]
To prove the claim of the theorem,
it hence remains to estimate $I_Q \psi_{E,Q_0}$.

Following \cite{BFHH},
we say that a parabolic cube $Q$ is of 
\begin{itemize}
  \item type 1, if $3Q \cap E \ne \varnothing $ and $3Q \subset Q_0$;
  \item type 2, if $3Q \cap E = \varnothing$ and $3Q \subset Q_0$;
  \item type 3, if $3Q \cap \frac{1}{2} Q_0 = \varnothing$;
  \item type 4, if it is of none of the other types.
\end{itemize} 
We will prove 
\begin{equation}
\label{eq:mean-osc-for-bmo-cases}
\inf_{c \in \R}  \int_{Q}|I_Q \psi_{E,Q_0}(\xbf) -c | \, d\xbf \le |Q| 
\end{equation} 
first in each of the three first cases separately, 
and then we derive \eqref{eq:mean-osc-for-bmo-cases} for type 4 cubes based on this information.

\textit{Case 1: Type 1 cubes.}
These cubes are such that their dilates by a factor (absolute constant) contain a Whitney cube of $E^{c}$ and they are well inside $Q_0$.
Consider a smooth test function $g \in C_c^{\infty}(Q)$ with $\no{g}_{L^{2}(Q)} \le 1$.
Let $P_r$ be a parabolic Littlewood--Paley operator at scale $r$.
We use the version with compact space support and rapidly decaying Fourier support.
Using the fundamental theorem of calculus,
we compute 
\begin{multline*}
\int \partial_t I_Q \psi_{E,Q_0}(0,\xbf) g(\xbf) \, d \xbf  
=   b_1 - \int_{0}^{\diam(Q)/(8MC_h)} \int   \partial_r [\partial_t I_Q \psi_{E,Q_0}(r,\xbf) P_r g(\xbf)] \, d \xbf dr \\
= b_1 - \int_{0}^{\diam(Q)/(8MC_h)} \int    \partial_t \partial_r \psi_{E,Q_0}(r,\xbf)  I_Q P_r g(\xbf)  \, d \xbf dr \\
 - \int_{0}^{\diam(Q)/(8MC_h)} \int     \partial_t  \psi_{E,Q_0}(r,\xbf)  \partial_r P_r  I_Q g(\xbf) \, d \xbf dr \\
= b_1 + b_2 - \int_{0}^{\diam(Q)/(8MC_h)} \int    r \partial_t \partial_r^{2} \psi_{E,Q_0}(r,\xbf)  I_Q P_r g(\xbf)  \, d \xbf dr \\
- \int_{0}^{\diam(Q)/(8MC_h)} \int    r \partial_t \partial_r \psi_{E,Q_0}(r,\xbf)  I_Q \partial_r P_r g(\xbf)  \, d \xbf dr \\
 - \int_{0}^{\diam(Q)/(8MC_h)} \int   r  \partial_t  \psi_{E,Q_0}(r,\xbf)  \partial_r P_r  I_Q g(\xbf) \, d \xbf dr \\
=: b_1 + b_2 - A 
\end{multline*}
where 
\begin{align*}
b_1 &= \int   \partial_t I_Q \psi_{E,Q_0}(\diam(Q)/(8MC_h),\xbf) P_{\diam(Q)/(8MC_h)} g(\xbf) \, d \xbf , \\
b_2 &= \left[ \int   \partial_r[\partial_t I_Q \psi_{E,Q_0}(r,\xbf) P_{r} g(\xbf)] \, d \xbf \right]_{r= \diam(Q)/(8MC_h)} .
\end{align*}

Now by Cauchy-Schwarz inequality,
$L^{2}$-estimates for parabolically anisotropic Littlewood--Paley square functions
and by the integral estimates from Lemma \ref{lemma:square-function}
\[
|A| \lesssim |Q|^{1/2}.
\]
To estimate $b_1$ and $b_2$,
we can (for instance),
pass the operator $\partial_t I_Q$ to act on $g$.
Then the desired bound follows from $|\psi_{E,Q_0}| \lesssim 1$ and
Young's convolution inequality.  

\textit{Case 2: Type 2 cubes.}
These cubes are such that $3Q \subset E^{c} \cap Q_0$.
Hence if $Q$ is of type 2,
then $h_E > 0$ in $3Q$. 
Consequently by Young's convolution inequality 
\[   \no{\partial_t I_Q \psi_{E,Q_0}}_{L^{\infty}(Q)}
= \no{I_Q  \partial_t\psi_{E,Q_0}}_{L^{\infty}(Q)} 
\le  2\no{K_Q}_{L^{1}(\rn)} \no{ \partial_t \psi_{E,Q_0}}_{L^{\infty}(3Q)}
\lesssim 1
\]
where the last estimate used the chain rule, implicit differentiation,
interior $L^{\infty}$ bounds on solutions and the definitions of $h_E$ and $K_Q$.
It is the last item that takes advantage of $h_E > 0$.
The estimate \eqref{eq:mean-osc-for-bmo-cases} follows.

\textit{Case 3: Type 3 cubes.}
By localization in the definition of $I_Q$, if $Q$ is of type 3, then $\psi_{E,Q_0} = 0$ in $3Q$ and \eqref{eq:mean-osc-for-bmo-cases} holds trivially.

\textit{Case 4: Type 4 cubes.} 
Because the non-local part of the half-order derivative $\partial_t I_2 \psi_{E,Q_0}$
is in $BMO$ and 
\[
\inf_{c \in \R} \int_{Q} |\partial_t I_Q \psi_{E,Q_0}(\xbf) - c| \,d \xbf \lesssim |Q|
\]
for all cubes of type 1, 2, or 3,
we can already consider proved that there exists $N$ independent of $Q$ such that   
for the full $\partial_t I = \partial_t (I_Q + I_2)$
\begin{equation}
\label{eq:3-first-types}
\inf_{c \in \R} \int_{Q} |\partial_t I \psi_{E,Q_0}(\xbf) - c| \,d \xbf \le N |Q|
\end{equation}
whenever $Q$ is of type 1, 2, or 3.
In particular, if $3Q$ and $3Q'$ are cubes like this with the additional properties $|Q| = |Q'|$ and $2Q \cap 2Q' \ne \varnothing$,
it holds $Q  \subset 3Q'$ and consequently by \eqref{eq:3-first-types}
\begin{equation}
\label{eq:type4-intermediate-1}
|(\partial_t I  \psi_{E,Q_0}(\xbf))_{Q} - (\partial_t I  \psi_{E,Q_0}(\xbf))_{Q'}|
\le 3^{n+1}N.
\end{equation}

Consider an integer $k > 0$ (for instance $k=3$, but in order to not mess the exponents for the time being, we keep it $k$) and the cube $Q_k = 2^{k} Q_0$.
We partition $Q_k $ into $N_k = 2^{(k+4)(n+1)}$ parabolic subcubes $\{Q^{i}\}_{i=1}^{N_{k}}$ 
with measure $2^{-4(n+1)}|Q_0|$.
These parabolic subcubes as well as the cubes $3Q^{i}$ are of type 1, 2, or 3. 
Then 
\begin{multline*}
\frac{1}{|Q^{1}|} \int_{Q_{k}} |\partial_t I  \psi_{E,Q_0}(\xbf) - (\partial_t I_Q \psi_{E,Q_0}(\xbf))_{Q^{1}}| \,d \xbf
\le \sum_{i=1}^{N_{k}} \int_{Q^{i}} |\partial_t I  \psi_{E,Q_0}(\xbf) - (\partial_t I_Q \psi_{E,Q_0}(\xbf))_{Q^{i}}| \,d \xbf  \\
+ |(\partial_t I  \psi_{E,Q_0}(\xbf))_{Q^{i}} - (\partial_t I  \psi_{E,Q_0}(\xbf))_{Q^{1}}|
\le N + 2^{k+4} 3^{n+1}Nn
\end{multline*}
For the first term we used that every $Q^{i}$ is of type 1, 2, or 3.
For the second term we could connect any $Q^{i}$ to $Q^{1}$ by $2^{k+4}n$ cubes and apply \eqref{eq:type4-intermediate-1}.

Now every type 4 cube $Q$ satisfies $\diam(Q) > \frac{1}{16} \diam(Q_0)$
and hence for every type 4 cube with $Q \subset Q_k$
\[
\inf_{c \in \R} \int_{Q} |\partial_t I \psi_{E,Q_0}(\xbf) - c| \, \xbf \le 16N (1+2^{k+4}3^{n+1}n)|Q|.
\]
If $Q$ is a type four cube with $Q \subsetneq Q_k$,
then we actually know from $Q \cap \frac{1}{2} Q_{0} \ne \varnothing$ that $\diam(Q) \ge 2^{k-2}\diam(Q_0)$.
We may cut $Q$ into $16^{n+1}$ subcubes $Q'_{j}$ so that $2^{n+1}$ of them cover $Q_0 \cap Q$. 
We can choose these cubes so that $3Q' \cap Q_0 = \varnothing$ for all the other cubes $Q_{j}'$,
thus ensuring that the remaining cubes are of type 3.

Let $Q'_1$ be one of the cubes with $Q'_{1} \cap Q_0 = \varnothing$.
Then setting $c_{j} = (\partial_t I  \psi_{E,Q_0}(\xbf))_{Q'_{j}}$,
we see that for $N' = 2^{k+10}3^{n+1}n N$
\begin{multline*}
|\{ \xbf \in Q: |\partial_t I  \psi_{E,Q_0}(\xbf) -  c_1| > N' \}|
\le \sum_{j=1}^{16^{n+1}}1_{\{Q_{j}' \cap Q_{0} = \varnothing\}}(j) |\{ \xbf \in Q'_{j}: |\partial_t I  \psi_{E,Q_0}(\xbf) -  c_j| > N'/2 \}| \\
+  \sum_{j=1}^{16^{n+1}}1_{\{Q_{j}' \cap Q_{0} = \varnothing\}}(j) |\{ \xbf \in Q'_{j}: |c_1 -  c_j| > N'/2 \}|
+ 2^{-3(n+1)}|Q| < \frac{1}{4}|Q|;
\end{multline*}  
indeed, by \eqref{eq:type4-intermediate-1} 
\[
1_{\{Q_{j}' \cap Q_{0} = \varnothing\}}(j)|c_1 -  c_j| \le 32  \cdot 3^{n+1} 2^{k+4} n N = N'/2 
\]
and by \eqref{eq:3-first-types}
\begin{multline*}
1_{\{Q_{j}' \cap Q_{0} = \varnothing\}}(j) |\{ \xbf \in Q'_{j}: |\partial_t I  \psi_{E,Q_0}(\xbf) -  c_j| > N'/2 \}| \\
\le \frac{2}{N'} 1_{\{Q_{j}' \cap Q_{0} = \varnothing\}}(j)  \int_{Q'_{j}} |\partial_t I  \psi_{E,Q_0}(\xbf) -  c_j| \, d \xbf
\le \frac{32|Q|N}{N'} < \frac{1}{8}|Q|.
\end{multline*}
The proof is now complete by the John--Str\"omberg characterization of $BMO$ \cite{JohnStrom}.
\end{proof}

\subsection{Existence of the base of a good sawtooth}
To conclude the proof,
it remains to show there exists a closed set $E \subset Q_0$ 
so that the assumptions of Lemma \ref{lemma:square-function} hold true.
To this end,
we recall the following lemma relating the values of the Green's function to the caloric measure.
For a proof, see \cite{HL-Mem,FGS, FS, FSY}.
\begin{lemma}
\label{lemma:green-caloric}
Let $f$ be parabolic $L$-Lipschitz. 
Let $Q_0 \subset \R^{n}$.
Let $\Xbf_0 = {\trm crk}_f(Q_0)$.
Let $G$ be the adjoint Green's function with pole at $\Xbf_0$.
Then for all $\Xbf \in \mathcal{R}(Q_0)$ 
\[
\frac{G(\Xbf)}{\dist(\Xbf,\Omega^{c})} \sim \frac{\omega^{\Xbf_0}(\Delta(\Xbf))}{\dist(\Xbf,\Omega^{c})^{n-1}}.
\]
The implicit constants here depend only on $n$ and $L$. Here and below we have set
\[\Omega = \Omega_f.\]
\end{lemma}

Lemma \ref{lemma:green-caloric} allows us to translate information about parabolic cubes 
related to a stopping time using the size of the caloric measure to bounds on the Green's function.
Ultimately this will ensure the sawtoowth above the contact set relative to the stopping time to be $M$-good.

\begin{lemma}
\label{lemma:stopping-lemma}
Let $f$ be parabolic $L$-Lipschitz. 
Let $Q_0 \subset \R^{n}$.
Let $\Xbf_0 = {\trm crk}_f(Q_0)$.
Let $\omega$ be the pulled-back caloric measure with pole at $\Xbf_0$,
as defined in \eqref{def:pullback-omega},
and assume that $\omega(Q_0) > 0$ and $\sigma \ll \omega$.

Let $\varepsilon > 0$.
Then there exists $M \ge 1$ such that if $\mathcal{F}$ is the family of maximal dyadic parabolic subrectangles $Q \subset Q_0$ with 
\[
\frac{1}{M} > \frac{\omega(Q)}{|Q|} \frac{|Q_0|}{\omega(Q_0)} \quad \trm{or} \quad M < \frac{\omega(Q)}{|Q|} \frac{|Q_0|}{\omega(Q_0)},
\]
then 
\[
\left \lvert \bigcup_{Q \in \mathcal{F}} Q \right \rvert \le \varepsilon |Q_0|.
\]
\end{lemma}

\begin{proof}
Because $\sigma \ll \omega$, 
for $\varepsilon > 0$ there exists $\delta > 0$ such that if $\omega(E) < \delta$,
then $|E| <|Q_0| \varepsilon /2$.
We set $M = 2 \max( \omega(Q_0)/\delta , 1 / \varepsilon )$.
Consider now the family $\mathcal{F}$ as in the statement.
Let 
\[
\mathcal{F}_1 = \{Q  \in \mathcal{F} :  M  < \omega(Q) |Q_0| / [\omega(Q_0)|Q|]\}
\] 
and $\mathcal{F}_2 = \mathcal{F} \setminus \mathcal{F}_1$. Then 
\[
 \left \lvert \bigcup_{Q \in \mathcal{F}_1} Q \right \rvert
\le \sum_{Q \in \mathcal{F}_1} |Q|
\le \frac{|Q_0|}{M \omega(Q_0)} \sum_{Q \in \mathcal{F}_1} \omega(Q)
\le \frac{|Q_0|}{M} < \frac{\varepsilon}{2} |Q_0|
\]
and 
\[
\omega \left(\bigcup_{Q \in \mathcal{F}_2} Q \right)
\le \sum_{Q \in \mathcal{F}_2} \omega(Q)
\le \frac{\omega(Q_0)}{M |Q_0|} \sum_{Q \in \mathcal{F}_2} |Q|
\le \frac{\omega(Q_0)}{M} < \delta .
\]
By absolute continuity,
we then conclude the claimed inequality. 
\end{proof}
 
\begin{proof}[Proof of (a) implies (c) in Theorem \ref{main.thrm-coord}] 
For each integer $k > 0$,
we apply Lemma \ref{lemma:stopping-lemma} in $4Q_0$ with $\varepsilon = 1/k$ to create a family $\mathcal{F}_k$.
Let 
\[
E_k = \overline{4Q_0} \setminus {\rm int} \left(\bigcup_{Q \in \mathcal{F}_k} Q\right).
\] 
Then for the closed set $E_k$,
we have the bound
\[
|E_k| \ge  |Q_0| - \abs{\bigcup_{Q \in \mathcal{F}_k} Q} \ge (1-\varepsilon_k)|Q_0|. 
\]
We claim that for every $\theta > 0$,
there exists $M = M_{k,\theta}$ so that $S_{E_k,\theta}$ is an $M$-good Green subdomain of $\Omega$ (according to Definition \ref{def:m-good-green}).
By Lemma \ref{lemma:m-goodness-and-slope}, this follows for all $\theta$ if it is true for any specific value of $\theta$.
However, if $Q \in \mathcal{F}_k$ and $\widehat{Q}$ is the parent cube of $Q$,
then by Lemma \ref{lemma:green-caloric} there exists $M'$ only depending on $M$, the dimension and $L$ such that 
\[
\frac{1}{M'} \le  G(\Xbf) \frac{|4Q_0|}{\omega(4Q_0)} \le M'
\]
for all $\Xbf \in W_f(\widehat{Q})$. 
But clearly there exists $\theta > 0 $ such that 
\[
S_{E_k,\theta}  \subset \bigcup_{Q \in \mathcal{F}} W(\widehat{Q}) .
\]
Hence $S_{E_k,\theta}$ is $M'$-good Green subdomain.
The assumptions of Lemma \ref{lemma:square-function} are satisfied, 
and tautologically the same goes for those of Lemma \ref{lemma:bmo-estimate}.
Hence $\psi_{E_k,Q_0}$ is a regular $Lip(1,1/2)$ function with 
\[
\sigma\left(Q_0 \setminus \{(\xbf,\psi_{E_k,Q_0}): \xbf \in   Q_0\} \right) \le \frac{1}{k}
\]
and consequently 
\[
\sigma\left(Q_0 \setminus \bigcup_{k=1}^{\infty} \{(\xbf,\psi_{E_k,Q_0}): \xbf \in  Q_0\} \right) = 0.
\]
\end{proof}

\section{Beta numbers imply absolute continuity}
\label{sec:b-to-a}

In this section we prove the implication from (b) to (a) in Theorem \ref{main.thrm-coord}.
Fix $Q_0$ and a compact set $K \subset Q_0$.
Forming the Whitney decomposition $\mathcal{W}$ of $K^{c}$ and considering the smooth parition of the unity $\{\varphi_Q : Q \in \mathcal{W}\}$,
we define the selective mollification 
\begin{equation}
\label{def:selective-mollification}
f_K = 1_{K} f + \sum_{\substack{Q \in \mathcal{W} \\ Q \cap 3Q_0 \ne \varnothing}} \varphi_Q f(c_Q) 
\end{equation}
where $c_Q$ is the center of $Q$ if $Q \cap 3Q_0 \ne \varnothing$ and $c_Q$ is the center of $Q_0$ otherwise.
Note that $h_E$ from \eqref{def:reg-distance} was essentially selective mollification of $ \xbf \mapsto \dist(\xbf,E)$.
For a non-zero multi-index $\alpha = (\alpha',\alpha_n) \in \N^{n}$ and a point $\xbf \notin K$,
we have the following estimate 
\begin{multline}
\label{eq:selective-mollitifcation-derivative}
|\partial^{\alpha} f_K (\xbf)|
= \abs{ \sum_{\substack{Q \in \mathcal{W} \\ Q \cap 3Q_0 \ne \varnothing}} \partial^{\alpha} \varphi_Q(\xbf) f(c_Q)  } \\
= \abs{ \sum_{\substack{Q \in \mathcal{W} \\ Q \cap 3Q_0 \ne \varnothing}} \partial^{\alpha} \varphi_Q(\xbf)[ f(c_Q) - f(\xbf) ] }
\lesssim \no{f}_{Lip(1,1/2)}\dist(\xbf,K)^{-(2\alpha_n +|\alpha'|)+1} .
\end{multline}

\begin{lemma}
\label{lemma:btoa-betasformollified}
Let $M > 0$.
If $f$ is parabolic $L$-Lipschitz and $K \subset Q_0$ a compact set such that 
\begin{equation}
\label{eq:btoalemmasq-assumption}
\sup_{\xbf \in K} \int_{0}^{1} \check{\beta}(f,\xbf;r)^{2} \frac{dr}{r} \le M,
\end{equation}   
then $f_K$ from \eqref{def:selective-mollification} is parabolic $CL$-Lipschitz and satisfies the Carleson measure estimate
\[
\int_{0}^{\ell(Q)} \int_{Q} \check{\beta}(f_K,\xbf;r)^{2} \frac{dr}{r} \le C(L + M')|Q| 
\]
for all parabolic cubes $Q$.
The constant $C$ only depends on the dimension and $M'$ only on $M$ and the dimension.
\end{lemma}

\begin{proof}
We first prove the parabolic Lipschitz condition.
Let $\xbf ,\ybf \in \rn$.
If $Q(\xbf,8|\xbf - \ybf|) \cap K = \varnothing$,
then for every $P \in \mathcal{W}(K)$ with $3P \cap Q(\xbf,|\xbf - \ybf|) \ne \varnothing$
we have that $\ell(P) \sim \dist(\xbf,K)$.
Consequently, 
denoting $\xbf = (x,t)$ and $\ybf = (y,s)$ and using \eqref{eq:selective-mollitifcation-derivative}
along with the fundamental theorem of calculus
\begin{multline*}
|f_K(\xbf) - f_K(\ybf)|  
\le |f_K(x,t) - f_K(x,s)| + |f_K(x,s) -   f_K(y,s)| \\
\lesssim L \left(\frac{|t-s|}{\dist(\xbf,K)} + |x-y|\right) 
\lesssim L \ab{\xbf-\ybf}.
\end{multline*} 
If on the other hand $Q(\xbf,8|\xbf - \ybf|) \cap K \ne \varnothing$,
then there exists $\zbf \in K$ such that 
\[
|\zbf - \xbf| + |\zbf - \ybf| \lesssim |\xbf - \ybf| .
\]
Now 
\[
|f_K(\xbf) - f_K(\ybf)|
\le 2\max(|f_K(\xbf) - f(\zbf)| ,|f(\zbf) - f_K(\ybf)|)  
\]
and we may assume without loss of generality that $|f_K(\xbf) - f(\zbf)|\ge|f(\zbf) - f_K(\ybf)|$.
By \eqref{def:selective-mollification} again,
we see that  
\[
|f_K(\xbf) - f(\zbf)|
\le \sum_{P \in \mathcal{W}(K)} \varphi_P(\xbf) |f(c_P) - f(\zbf)| \lesssim L |\xbf - \zbf| \lesssim L |\xbf - \ybf|.
\] 
This concludes the proof that $f_K$ is $Lip(1,1/2)$ with the claimed constant.

Next we turn to the heart of the matter.
Fix a point $\xbf \in \rn$ and a sidelength $\rho > 0$.
We have to prove a uniform bound for 
\begin{equation}
\label{eq:btoa-toestimate}
\frac{1}{\rho^{n+1}} \int _{0}^{1} \int_{Q(\xbf,\rho)} \check{\beta}(f_K,\ybf;\rho)^{2} \frac{d\ybf dr}{r}
\end{equation}
and we will do so using an argument essentially from \cite{DS-Ast,BHHLN-CME}.
Denote 
\[
\delta_{\rho}(\ybf) := \min (\rho , \dist(\ybf,K)) .
\]
We split the integral in \eqref{eq:btoa-toestimate}
in two parts  
\begin{align*}
\I &:= \int_{Q(\xbf,\rho) \setminus K} \int_{0}^{\delta_{\rho}(\ybf)/60} \check{\beta}(f_K,\ybf;r)^{2} \frac{dr d\ybf }{r} , \\
\II &:=  \int_{Q(\xbf,\rho)} \int_{\delta_{\rho}(\ybf)/60}^{\rho} \check{\beta}(f_K,\ybf;r)^{2} \frac{dr d\ybf }{r}.
\end{align*}

We start with the easier term $\I$,
whose estimation only requires the Lipschitz condition but not \eqref{eq:btoalemmasq-assumption}.
Given $\ybf = (y,t_0)$ and $r < \delta_{\rho}(\ybf)/60$,
we have 
\begin{multline*}
\check{\beta}(f_K,\ybf;r) =  \inf_{A} \left(\fint_{Q(\ybf,r)} \left( \frac{|f_K(\zbf) - A(z)|}{r}\right)^{2} \, d\zbf \right)^{1/2}
\le  \left(\fint_{Q(\ybf,r)} \left( \frac{|f_K(\zbf) - f_K(z,t_0)|}{r}\right)^{2} \, d\zbf \right)^{1/2}\\
+ \inf_{A} \left( \fint_{Q(\ybf,r)} \left( \frac{|f_K(z,t_0) - A(z)|}{r}\right)^{2} \, dz \right)^{1/2}
= \I_1 + \I_2.
\end{multline*}
Clearly there are at most $c_n$ cubes $P \in \mathcal{W}(K)$ with $3P \cap Q(\ybf,r)$,
by the fact that $60 r < \dist(\ybf,K)$.
Also,
for those $P$ it holds $\ell(P) \sim \dist(\ybf,K)$.
Hence we can estimate by \eqref{eq:selective-mollitifcation-derivative}
\[
I_1 \le \frac{1}{r} \sup_{\{z: |z-y| < 2r\}} \int_{t_0-4r^{2}}^{t_0+4r^{2}}|\partial_t f_K(z,t)| \, dt \lesssim r \sup_{\{z: |z-y| < 2r\}} \  \max_{\{t: |t-t_0| \le 4r^{2}\} } |\partial_t f_K(z,t)| \lesssim \frac{Lr}{\dist(\ybf,K)} .
\]
To estimate $I_2$,
we apply Poincar\'e's inequality with second derivatives (or Taylor's theorem) in spatial variables and \eqref{eq:selective-mollitifcation-derivative} to estimate 
\[
\I_2 \lesssim r \sup_{\zbf \in Q(\ybf,r)} |\nabla^{2} f_{K}(\zbf)| \lesssim \frac{Lr}{\dist(\ybf,K)}.
\]
Using these estimates,
we see that 
\[
\I \lesssim L^2\int_{Q(\xbf,\rho) \setminus K} \int_{0}^{\delta_{\rho}(\ybf)/60} \frac{r}{\dist(\ybf,K)^{2}} \, dr d\ybf 
  \lesssim   L^2 \rho^{n+1}
\]
as was claimed. 

We turn to the Marcinkiewicz integral type term $\II$,
where the information \eqref{eq:btoalemmasq-assumption} will be important.
Consider $\ybf = (y,t_0)$ and $r \ge \delta_{\rho}(\ybf)/60$.
We estimate  
\begin{multline*}
\check{\beta}(f_K,\ybf;r) =  \inf_{A} \left(\fint_{Q(\ybf,r)} \left( \frac{|f_K(\zbf) - A(z)|}{r}\right)^{2} \, d\zbf \right)^{1/2}
\le \inf_{A} \left(\fint_{Q(\ybf,r)} \left( \frac{|f(\zbf) - A(z)|}{r}\right)^{2} \, d\zbf \right)^{1/2} \\
+  \left(\fint_{Q(\ybf,r)} \left( \frac{|f(\zbf)-f_K(\zbf)|}{r}\right)^{2} \, d\zbf \right)^{1/2} = \II_1 + \II_2.
\end{multline*}
Assume first $r \ge \dist(y,K)/60$.
Then there exists $\ybf^{*} \in K$ such that 
\[
Q(\ybf,r) \subset Q(\ybf^{*},180r).
\]  
Now 
\[
\II_1 \lesssim \inf_{A} \left(\fint_{Q(\ybf^{*},180r)} \left( \frac{|f(\zbf) - A(z)|}{r}\right)^{2} \, d\zbf \right)^{1/2}
  = \check{\beta}(f,\ybf^{*};180r)
\]
and by the Lipschitz condition on $f$
\begin{multline*}
\II_2^{2}  = \fint_{Q(\ybf,r)} \left( \frac{|f(\zbf)-f_K(\zbf)|}{r}\right)^{2} \, d\zbf  
  \lesssim L^{2} \fint_{Q(\ybf,r)} \left( \sum_{P \in \mathcal{W}(K)} \varphi_{P}(\zbf) \dist(z,K) \right)^{2} \, d\zbf \\
  \lesssim L^{2} \fint_{Q(\ybf,r)} \left( \frac{\dist(\zbf,K)}{r} \right)^{2} \, d\zbf.
\end{multline*}
Finally, if $r \ge \rho/60$,
then $\II_1 + \II_2 \lesssim L$.

Hence by these estimates and by \eqref{eq:btoalemmasq-assumption}
\[
\II \lesssim M' \rho^{n+1} + 
L^{2} \int_{Q(\xbf,\rho)} 
\int_{d(\ybf,K)/60}^{\rho}
\int_{Q(\ybf,r)} \left( \frac{\dist(\zbf,K)}{r} \right)^{2} \frac{d\zbf dr  d\ybf }{r^{n+2}}.
\]
In the second term,
we use that for $r$, $\ybf$ and $\zbf$ as in the domain of integration 
\[
60 r \ge \dist(\ybf,K) \ge \dist(\zbf,K) - r 
\] 
so that using 
\[ 
r \ge \dist(\zbf,K)/61, \quad |\zbf - \ybf| \le r
\]
and finally by Tonelli's theorem we get an upper bound by 
\begin{multline*}
L^{2} \int_{Q(\xbf,\rho)} 
 \int_{d(\zbf,K)/61}^{\rho} \int_{Q(\zbf,r)} \left( \frac{\dist(\zbf,K)}{r} \right)^{2} \frac{d\ybf  dr  d\zbf }{r^{n+2}} \\
\lesssim L^{2} \int_{Q(\xbf,\rho)} 
  \int_{d(\zbf,K)/61}^{\rho} \left( \frac{\dist(\zbf,K)}{r} \right)^{2} \frac{ dr  d\zbf }{r}
\lesssim L^{2} \rho^{n+1}.
\end{multline*}
This concludes the proof.
\end{proof}

\begin{proof}[Proof of (b) implies (a) in Theorem \ref{main.thrm-coord}] 
Fix a corkscrew point $\Xbf_0 = (X_0,t_0)$.
Let $E \subset \R^{n}$ be a Borel set such that 
\[ 
|\{ (X,t) \in E : t < t_0 \}| > 0.
\]
Our goal is to prove that 
\[ 
\omega^{\Xbf_0}(\{ (X,t) \in E : t < t_0 \}) > 0 .
\]
Let $\varepsilon \in (0,1)$.
By the Lebesgue differentiation theorem,
there exists a cube $Q_0 \subset \{(x,t): t < t_0\}$
such that $|E\cap Q_0| \ge (1 - \varepsilon^{2})|Q_0|$.
By the hypothesis of the theorem and Proposition \ref{prop:betas-pullback} 
\[
|Q_0| = \abs{\left \lbrace \xbf \in Q_0 : \int_{0}^{1} \check{\beta}(\xbf;r)^{2} \frac{dr}{r} < \infty  \right \rbrace}
= \lim_{M \to \infty} \abs{\left \lbrace \xbf \in Q_0 : \int_{0}^{1} \check{\beta}(\xbf;r)^{2} \frac{dr}{r} < M  \right \rbrace}.
\]
By inner regularity of the Lebesgue measure we find $M > 0$ and a compact set $K \subset Q_0 \cap E$ such that 
\[
|Q_0 \cap E \setminus K| < \varepsilon, \quad \sup_{\xbf \in K} \int_{0}^{1} \check{\beta}(\xbf;r)^{2} \frac{dr}{r} \le M.
\]
By Lemma \ref{lemma:btoa-betasformollified}, the function $f_{K}$ defined through \eqref{def:selective-mollification} based on this compact $K$ is a regular $Lip(1,1/2)$ function.
Because 
\[
|f_K(\xbf) - f(\xbf)| \le L \dist(\xbf,\Omega) \lesssim L h_K(\xbf) 
\]
where $h_K$ is the regularized distance \eqref{def:reg-distance} from $K$ (or equivalently the selective mollification of the distance from that set),
there is constant $c_n$ such that the function 
\[
\tilde{f}(\xbf) = f_K(\xbf) + c_n L h_K(\xbf)
\]  
is a regular $Lip(1,1/2)$ function satisfying 
\[
\tilde{f}(\xbf) \ge f(\xbf), \quad 1_{K}(\xbf)\tilde{f}(\xbf) = 1_{K}(\xbf)f(\xbf) .
\]
Hence $\Omega_{\tilde{f}} \subset \Omega_{f}$.
Let $\mu^{\Xbf_0}$ be the pull-back along $\xbf \mapsto (\tilde{f}(\xbf),\xbf)$ of the caloric measure of $\Omega_{\tilde{f}}$ with pole at $\Xbf_0$.
By \cite{Lew-Mur-Mem}, the Lebesgue measure and $\mu^{\Xbf_0}$ are mutually absolutely continuous in $Q_0$.
Hence $\mu^{\Xbf_0}(K) > 0$.
Now the standard argument using the fact that the caloric measures are bounded in $\Xbf_0$ variable and the maximum principle 
gives that
\[
\omega^{\Xbf_0}(E)  \ge \omega^{\Xbf_0}(K) \ge  \mu^{\Xbf_0}(K) > 0
\]
which concludes the proof.
\end{proof}

\begin{proof}[Proof of Theorem \ref{partconv.thrm}] 
This is almost a copy of the previous one.
By the reductions in subsection \ref{sec:reductions},
it suffices to prove a version pulled back to $\rn$.
Fix a corkscrew point $\Xbf_0 = (X_0,t_0)$.
Let $E \subset \R^{n}$ be a Borel set such that 
\[ 
\omega^{\Xbf_0}(\{ (X,t) \in E : t < t_0 \}) > 0 .
\]
Our goal is to prove that 
\[ 
|\{ (X,t) \in E : t < t_0 \}| > 0.
\]
Let $\varepsilon \in (0,1)$.
By the Radon--Nikodym theorem,
there exists a cube $Q_0 \subset \{(x,t): t < t_0\}$
such that $\omega^{\Xbf_0}(E\cap Q_0) \ge (1 - \varepsilon^{2})\omega^{\Xbf_0}(Q_0)$.
By the hypothesis of the theorem and Proposition \ref{prop:betas-pullback} 
\[
\omega^{\Xbf_0}(Q_0) = \omega^{\Xbf_0}\left(\left \lbrace \xbf \in Q_0 : \int_{0}^{1} \check{\beta}(\xbf;r)^{2} \frac{dr}{r} < \infty  \right \rbrace\right)
= \lim_{M \to \infty} \omega^{\Xbf_0}\left(\left \lbrace \xbf \in Q_0 : \int_{0}^{1} \check{\beta}(\xbf;r)^{2} \frac{dr}{r} < M  \right \rbrace \right).
\]
By inner regularity of the caloric we find $M > 0$ and a compact set $K \subset Q_0 \cap E$ such that 
\[
\omega^{\Xbf_0}(Q_0 \cap E \setminus K) < \varepsilon, \quad \sup_{\xbf \in K} \int_{0}^{1} \check{\beta}(\xbf;r)^{2} \frac{dr}{r} \le M.
\]
By Lemma \ref{lemma:btoa-betasformollified}, the function $f_{K}$ defined through \eqref{def:selective-mollification} based on this compact $K$ is a regular $Lip(1,1/2)$ function.
Because 
\[
|f_K(\xbf) - f(\xbf)| \le L \dist(\xbf,\Omega) \lesssim L h_K(\xbf) 
\]
where $h_K$ is the regularized distance \eqref{def:reg-distance} from $K$ (or equivalently the selective mollification of the distance from that set),
there is constant $c_n$ such that the function 
\[
\tilde{f}(\xbf) = f_K(\xbf) - c_n L h_K(\xbf)
\]  
is a regular $Lip(1,1/2)$ function satisfying 
\[
\tilde{f}(\xbf) \le f(\xbf), \quad 1_{K}(\xbf)\tilde{f}(\xbf) = 1_{K}(\xbf)f(\xbf) .
\]
Hence $\Omega_{\tilde{f}} \supset \Omega_{f}$,
that is, in difference to the subdomain in the previous proof, we now construct a superdomain.
Let $\mu^{\Xbf_0}$ be the pull-back along $\xbf \mapsto (\tilde{f}(\xbf),\xbf)$ of the caloric measure of $\Omega_{\tilde{f}}$ with pole at $\Xbf_0$.
Now the standard argument using the fact that the caloric measures are bounded in $\Xbf_0$ variable and the maximum principle 
gives that
\[
0 < \omega^{\Xbf_0}(K) \le  \mu^{\Xbf_0}(K) . 
\]
By \cite{Lew-Mur-Mem}, the Lebesgue measure and $\mu^{\Xbf_0}$ are mutually absolutely continuous in $Q_0$
so $|E| \ge |K| > 0$, which concludes the proof.
\end{proof}

\section{Exhaustion implies beta numbers}
\label{sec:c-to-b}
In this section,
we prove the implication from (c) to (b) of Theorem \ref{main.thrm-coord}.
By assumption,
there exists regular Lipschitz functions $\{f_j\}_{j=1}^{\infty}$ such that 
\[
\abs{\R^{n} \setminus \bigcup_{j=1}^{\infty} E_j} = 0, \quad E_j := \{\xbf \in \rn: f(\xbf) = f_j(\xbf) \}.
\] 
Fix $j$ and $\xbf \in E_j$.
Given $\zbf \in Q(\xbf,2)$,
let $\zbf_j \in E_j$ be a point with $|\zbf_j - \zbf| = \dist(\zbf,E_j)$.
Such a point exists as $E_j$ is closed.
Denote $L_j = \no{f_j}_{Lip(1,1/2)} + \no{f}_{Lip(1,1/2)}$.
For any function $A : \R^{n} \to \R$
\[
|f(\zbf) - A(\zbf)| 
\le |f(\zbf) - f(\zbf_j)|+
|f(\zbf_j) - f_j(\zbf)|+
|f_j(\zbf) - A(\zbf)| 
\le L_j \dist(\zbf,E_j) + |f_j(\zbf) - A(\zbf)|.
\]
Now for $r \in (0,1)$
\[
\check{\beta}(f,\xbf;r)^{2} \le 2 \check{\beta}(f_j,x;r)^{2} + 2L_j^{2}\int_{Q(\xbf,r)} \left(\frac{\dist(\zbf,E_j)}{r} \right)^{2} \frac{\d\zbf}{r^{n}}.
\] 
Hence for any cube $Q_0$ with radius $1$
\begin{equation}
\label{eq:rect-proof-betaj}
\int_{Q_0 \cap E_j} \left( \int_{0}^{1}\check{\beta}(f,\xbf;r)^{2} \frac{dr}{r} \right) \, d\xbf \le c_j(1+\I)
\end{equation}
where 
\[
\I = \int_{Q_0 \cap E_j}  \int_{0}^{1}  \int_{Q(\xbf,r) \setminus E_j} \left(\frac{\dist(\zbf,E_j)}{r} \right)^{2} \frac{d\zbf dr d \xbf}{r^{n+1}}.
\] 
Using that $r \ge |\xbf  - \zbf| \ge \dist(\zbf,E_j)$, we obtain an upper bound by 
\[
c_j +  \int_{Q_0 \cap E_j} \int_{Q(\xbf,1) \setminus E_j}  \frac{d\zbf d \xbf}{|\xbf - \zbf|^{n}}
< \infty.
\]   
For the last inequality, 
we recall that the homogeneous dimension of $\R^{n}$ with parabolic metric is $n+1$.
The left hand side of \eqref{eq:rect-proof-betaj} being finite for all cubes $Q_0$ of radius 1
implies that for all $j$ and almost every $\xbf \in E_j$
\[
\int_{0}^{1}\check{\beta}(f,\xbf;r)^{2} \frac{dr}{r}  < \infty.
\]
Because the countable family $\{E_j\}_{j=1}^{\infty}$ covers almost all of $\rn$,
the proof is complete.

\bibliography{QBHMNrefs}

\newcommand{\etalchar}[1]{$^{#1}$}
\begin{thebibliography}{BHH{\etalchar{+}}23b}

\bibitem[AAM19]{AAM}
Murat Akman, Jonas Azzam, and Mihalis Mourgoglou.
\newblock Absolute continuity of harmonic measure for domains with lower
  regular boundaries.
\newblock {\em Adv. Math.}, 345:1206--1252, 2019.

\bibitem[ABHM17]{ABaHM}
Murat Akman, Matthew Badger, Steve Hofmann, and Jos\'e{}~Mar\'ia Martell.
\newblock Rectifiability and elliptic measures on 1-sided {NTA} domains with
  {A}hlfors-{D}avid regular boundaries.
\newblock {\em Trans. Amer. Math. Soc.}, 369(8):5711--5745, 2017.

\bibitem[ABHM19]{ABHM}
Murat Akman, Simon Bortz, Steve Hofmann, and Jos\'e{}~Mar\'ia Martell.
\newblock Rectifiability, interior approximation and harmonic measure.
\newblock {\em Ark. Mat.}, 57(1):1--22, 2019.

\bibitem[AHM{\etalchar{+}}16]{AHMMMTV}
Jonas Azzam, Steve Hofmann, Jos\'e{}~Mar\'ia Martell, Svitlana Mayboroda,
  Mihalis Mourgoglou, Xavier Tolsa, and Alexander Volberg.
\newblock Rectifiability of harmonic measure.
\newblock {\em Geom. Funct. Anal.}, 26(3):703--728, 2016.

\bibitem[AT15]{AT-partII}
Jonas Azzam and Xavier Tolsa.
\newblock Characterization of {$n$}-rectifiability in terms of {J}ones' square
  function: {P}art {II}.
\newblock {\em Geom. Funct. Anal.}, 25(5):1371--1412, 2015.

\bibitem[BFHPH25]{BFHH}
Simon Bortz, Sandra Ferris, Pablo Hidalgo-Palencia, and Steve Hofmann.
\newblock A variable coefficient free boundary problem for $l^p$-solvability of
  parabolic {Dirichlet} problems in graph domains.
\newblock Preprint. arXiv:2503.00873, 2025.

\bibitem[BHH{\etalchar{+}}22]{BHHLN-BPapprox}
Simon Bortz, John Hoffman, Steve Hofmann, Jose~Luis Luna-Garcia, and Kaj
  Nystr\"om.
\newblock On big pieces approximations of parabolic hypersurfaces.
\newblock {\em Ann. Fenn. Math.}, 47(1):533--571, 2022.

\bibitem[BHH{\etalchar{+}}23a]{BHHLN-corona}
S.~Bortz, J.~Hoffman, S.~Hofmann, J.~L. Luna-Garcia, and K.~Nystr\"om.
\newblock Corona decompositions for parabolic uniformly rectifiable sets.
\newblock {\em J. Geom. Anal.}, 33(3):Paper No. 96, 67, 2023.

\bibitem[BHH{\etalchar{+}}23b]{BHHLN-CME}
Simon Bortz, John Hoffman, Steve Hofmann, Jos\'e{}~Luis Luna~Garc\'ia, and Kaj
  Nystr\"om.
\newblock Carleson measure estimates for caloric functions and parabolic
  uniformly rectifiable sets.
\newblock {\em Anal. PDE}, 16(4):1061--1088, 2023.

\bibitem[BHMN25]{BHMN1}
Simon Bortz, Steve Hofmann, Jos\'e{}~Mar\'ia Martell, and Kaj Nystr\"om.
\newblock Solvability of the {${\rm L}^p$} {D}irichlet problem for the heat
  equation is equivalent to parabolic uniform rectifiability in the case of a
  parabolic {L}ipschitz graph.
\newblock {\em Invent. Math.}, 239(1):165--217, 2025.

\bibitem[CFK81]{CFK}
Luis~A. Caffarelli, Eugene~B. Fabes, and Carlos~E. Kenig.
\newblock Completely singular elliptic-harmonic measures.
\newblock {\em Indiana Univ. Math. J.}, 30(6):917--924, 1981.

\bibitem[Dah77]{Dahl-L2}
Bj\"orn E.~J. Dahlberg.
\newblock Estimates of harmonic measure.
\newblock {\em Arch. Rational Mech. Anal.}, 65(3):275--288, 1977.

\bibitem[DJ90]{DJ}
G.~David and D.~Jerison.
\newblock Lipschitz approximation to hypersurfaces, harmonic measure, and
  singular integrals.
\newblock {\em Indiana Univ. Math. J.}, 39(3):831--845, 1990.

\bibitem[DS91]{DS-Ast}
G.~David and S.~Semmes.
\newblock Singular integrals and rectifiable sets in {${\bf R}^n$}: {B}eyond
  {L}ipschitz graphs.
\newblock {\em Ast\'erisque}, (193):152, 1991.

\bibitem[EG82]{EG}
Lawrence~C. Evans and Ronald~F. Gariepy.
\newblock Wiener's criterion for the heat equation.
\newblock {\em Arch. Rational Mech. Anal.}, 78(4):293--314, 1982.

\bibitem[ENV25]{ENV}
Nick Edelen, Aaron Naber, and Daniele Valtorta.
\newblock Quantitative {R}eifenberg theorem for measures.
\newblock {\em Math. Z.}, 310(3):Paper No. 45, 70, 2025.

\bibitem[FGS84]{FGS}
Eugene~B. Fabes, Nicola Garofalo, and Sandro Salsa.
\newblock Comparison theorems for temperatures in noncylindrical domains.
\newblock {\em Atti Accad. Naz. Lincei Rend. Cl. Sci. Fis. Mat. Nat. (8)},
  77(1-2):1--12, 1984.

\bibitem[FJK84]{FJK}
Eugene~B. Fabes, David~S. Jerison, and Carlos~E. Kenig.
\newblock Necessary and sufficient conditions for absolute continuity of
  elliptic-harmonic measure.
\newblock {\em Ann. of Math. (2)}, 119(1):121--141, 1984.

\bibitem[FS97]{FS}
E.~B. Fabes and M.~V. Safonov.
\newblock Behavior near the boundary of positive solutions of second order
  parabolic equations.
\newblock In {\em Proceedings of the conference dedicated to {P}rofessor
  {M}iguel de {G}uzm\'an ({E}l {E}scorial, 1996)}, volume~3, pages 871--882,
  1997.

\bibitem[FSY99]{FSY}
E.~B. Fabes, M.~V. Safonov, and Yu~Yuan.
\newblock Behavior near the boundary of positive solutions of second order
  parabolic equations. {II}.
\newblock {\em Trans. Amer. Math. Soc.}, 351(12):4947--4961, 1999.

\bibitem[GH20]{GH-Ainf}
Alyssa Genschaw and Steve Hofmann.
\newblock A weak reverse {H}\"older inequality for caloric measure.
\newblock {\em J. Geom. Anal.}, 30(2):1530--1564, 2020.

\bibitem[HKM25]{HKM-sing}
Max Hallgren, Robert Koirala, and Zilu Ma.
\newblock Structure theory of parabolic nodal and singular sets.
\newblock Preprint. arXiv:2511.11570, 2025.

\bibitem[HL96]{HL-ann}
Steve Hofmann and John~L. Lewis.
\newblock {$L^2$} solvability and representation by caloric layer potentials in
  time-varying domains.
\newblock {\em Ann. of Math. (2)}, 144(2):349--420, 1996.

\bibitem[HL01]{HL-Mem}
Steve Hofmann and John~L. Lewis.
\newblock The {D}irichlet problem for parabolic operators with singular drift
  terms.
\newblock {\em Mem. Amer. Math. Soc.}, 151(719):viii+113, 2001.

\bibitem[HLN03]{HLN1}
Steve Hofmann, John~L. Lewis, and Kaj Nystr\"om.
\newblock Existence of big pieces of graphs for parabolic problems.
\newblock {\em Ann. Acad. Sci. Fenn. Math.}, 28(2):355--384, 2003.

\bibitem[HLN04]{HLN2}
Steve Hofmann, John~L. Lewis, and Kaj Nystr\"om.
\newblock Caloric measure in parabolic flat domains.
\newblock {\em Duke Math. J.}, 122(2):281--346, 2004.

\bibitem[Hof97]{Hof-SIO}
Steve Hofmann.
\newblock Parabolic singular integrals of {C}alder\'on-type, rough operators,
  and caloric layer potentials.
\newblock {\em Duke Math. J.}, 90(2):209--259, 1997.

\bibitem[KW80]{KW-counter}
Robert Kaufman and Jang~Mei Wu.
\newblock Singularity of parabolic measures.
\newblock {\em Compositio Math.}, 40(2):243--250, 1980.

\bibitem[LM95]{Lew-Mur-Mem}
John~L. Lewis and Margaret A.~M. Murray.
\newblock The method of layer potentials for the heat equation in time-varying
  domains.
\newblock {\em Mem. Amer. Math. Soc.}, 114(545):viii+157, 1995.

\bibitem[LS88]{Lew-Sil}
John~L. Lewis and Judy Silver.
\newblock Parabolic measure and the {D}irichlet problem for the heat equation
  in two dimensions.
\newblock {\em Indiana Univ. Math. J.}, 37(4):801--839, 1988.

\bibitem[NS17]{NS}
Kaj Nystr\"om and Martin Str\"omqvist.
\newblock On the parabolic {L}ipschitz approximation of parabolic uniform
  rectifiable sets.
\newblock {\em Rev. Mat. Iberoam.}, 33(4):1397--1422, 2017.

\bibitem[Str76]{JohnStrom}
Jan-Olov Str\"omberg.
\newblock Bounded mean oscillation with {O}rlicz norms and duality of {H}ardy
  spaces.
\newblock {\em Bull. Amer. Math. Soc.}, 82(6):953--955, 1976.

\bibitem[Tol15]{Tolsa-partI}
Xavier Tolsa.
\newblock Characterization of {$n$}-rectifiability in terms of {J}ones' square
  function: part {I}.
\newblock {\em Calc. Var. Partial Differential Equations}, 54(4):3643--3665,
  2015.

\bibitem[Wu79]{Wu-split}
Jang Mei~G. Wu.
\newblock On parabolic measures and subparabolic functions.
\newblock {\em Trans. Amer. Math. Soc.}, 251:171--185, 1979.

\bibitem[Wu86]{Wu-CADapprox}
Jang-Mei Wu.
\newblock On singularity of harmonic measure in space.
\newblock {\em Pacific J. Math.}, 121(2):485--496, 1986.

\end{thebibliography}
\bibliographystyle{alpha}
\end{document}